\documentclass[a4paper, 11pt, oneside]{amsart}
\usepackage{amsmath,amscd}
\usepackage{amssymb}
\usepackage{amsthm}
\usepackage{comment}
\usepackage{graphicx, xcolor}
\usepackage[abbrev,nobysame,alphabetic]{amsrefs}
\usepackage{mathrsfs}
\usepackage[ocgcolorlinks, linkcolor=blue]{hyperref}

\usepackage{bm}
\usepackage{bbm}
\usepackage{url}

\usepackage[utf8]{inputenc}
\usepackage{mathtools,amssymb}
\usepackage{esint}
\usepackage{tikz}
\usepackage{dsfont}
\usepackage{relsize}
\usepackage{url}
\usepackage{xcolor}
\usepackage{graphicx}
\usepackage{mathrsfs}
\usepackage[shortlabels]{enumitem}
\usepackage{lineno}
\usepackage{amsmath}
\usepackage{enumitem}
\usepackage{amsthm} 
\usepackage{verbatim}
\usepackage{dsfont}
\numberwithin{equation}{section}

\allowdisplaybreaks

\graphicspath{{images/}}

\newtheorem{thm}{Theorem}[section]

\newtheorem{lem}[thm]{Lemma}
\newtheorem{prop}[thm]{Proposition}

\theoremstyle{definition}
\newtheorem{rem}[thm]{Remark}

\newtheorem*{claim*}{Claim}
\numberwithin{equation}{section}

\newcommand{\R}{{\mathbb R}}

\newcommand{\LC}{\left(}
\newcommand{\RC}{\right)}

\newcommand{\norm}[1]{\lVert #1 \rVert}
\DeclareMathOperator{\supp}{supp} 

\newcommand{\dif}[1]{\,\mathrm{d}{#1}} 

\title[Fractional inverse scattering]{Inverse scattering for the fractional Schr\"{o}dinger equation}

\author{Saumyajit Das}
\address{Harish-Chandra Research Institute, A CI of Homi Bhabha National Institute, Chhatnag Road, Jhunsi, Prayagraj (Allahabad) 211 019, India}
\email{saumyajit.math.das@gmail.com}

\author{Tuhin Ghosh}
\address{Harish-Chandra Research Institute, A CI of Homi Bhabha National Institute, Chhatnag Road, Jhunsi, Prayagraj (Allahabad) 211 019, India}
\email{tuhinghosh@hri.res.in}

\author{Shiqi Ma}
\address{School of Mathematics, Jilin University, Changchun 130012, China}
\email{mashiqi@jlu.edu.cn}

\begin{document}

\begin{abstract}
	This article is devoted to studying the inverse scattering for the fractional Schr\"{o}dinger equation, and in particular we solve the Born approximation problem. Based on the ($p$,$q$)-type resolvent estimate for the fractional Laplacian, we derive an expression for the scattering amplitude of the scattered solution of the fractional Schr\"{o}dinger equation. We prove the uniqueness of the potential using the scattering amplitude data.
	
	\medskip
	
	
	
\end{abstract}

\maketitle


\section{Introduction}

\subsection{Problem setting and the results}

The inverse problem of finding a potential from its scattering amplitude is a subject of interest in mathematical physics and related fields due to its practicality. The mathematical richness of the topic is also compelling. For a detailed study, we refer the readers to \cites{GE, cakoni2006qualitative, colton1998inverse, isozaki2014recent, uhlmann2000inverse}.  In this article, we are concerned with the potential associated with the fractional Laplacian $(-\Delta)^s$ for $s > 0$. Various aspects of the fractional Laplacian \cite{CS07}, particularly for understanding Levy processes \cite{BSW} random walk analogies, and including the associated inverse problems \cite{GSU20} have been extensively studied in recent years.

There are several different definitions of $(-\Delta)^s$, see e.g.~\cite{kwasnicki2017ten}, some of which are not equivalent.
In this work, we adopt the definition by the Fourier transform $\mathcal F^{-1} |\xi|^{2s} \mathcal F$, and study inverse scattering problems related to $(-\Delta)^s$.
Hereafter, $\mathcal F$ and $\mathcal F^{-1}$ signify the Fourier and inverse Fourier transform, respectively.
We shall specify the domain of the definition of $(-\Delta)^s$ later in Section \ref{subsec:Ds}.

Let $d \geq 2$ be the dimension, $k > 0$ be the wave number, and $V \in L^{\infty}(\mathbb{R}^d)$ with compact support.
Assume $u$ satisfies the following fractional Schr\"{o}dinger equation: 
\begin{equation} \label{eq:1}
	\big( (-\Delta)^s - k^{2s} - V(x) \big) u = 0 \quad \mathrm{in} \quad \mathbb{R}^d.
\end{equation}
Suppose $k^{2s}$ is not in the point spectrum of $(-\Delta)^s - V$ in $\mathbb R^d$.
To make it meaningful, we excite the system \eqref{eq:1} by certain incident wave $u^{\mathrm{in}}$, for example, we could use the plane wave
\begin{equation} \label{eq:uin}
	u^{\mathrm{in}}(x) = e^{ikx \cdot \theta}, \quad \theta \in \mathbb{S}^{d-1}.
\end{equation}
It has been checked in \cite{LLM21}*{Lemma 3.1} that the plane wave $e^{ikx \cdot \theta}$ satisfies \eqref{eq:1} with $V = 0$ in the distributional sense.
One may also consider the Herglotz wave function defined as the linear superposition of plane waves with a density function $g \in L^{2}(\mathbb{S}^{d-1})$:
\begin{equation*} 
\int_{\mathbb{S}^{d-1}} e^{\mathsf{i}kx\cdot \theta} g(\theta) \dif s(\theta)
\end{equation*}
as the incident field $u^{\mathrm{in}}$.
The presence of the potential $V$ will generate a corresponding scattered wave
\begin{equation} \label{eq:usc}
u^{\mathrm{sc}} := u - u^{\mathrm{in}}.
\end{equation}
We define the outgoing scattered field by the outgoing limiting resolvent. More precisely, for a source term $f$ generated by the interaction with the potential, the free outgoing resolvent is understood as
\begin{equation} \label{eq:Rks}
	\mathcal R^s_{k^{2s}+\mathrm{i}0} f
	:= \lim_{\varepsilon\to 0^+} \big((-\Delta)^s-(k^{2s}+\mathrm{i}\varepsilon)\big)^{-1} f.
\end{equation}
The word outgoing here and below refers to this limiting absorption prescription.
Using \eqref{eq:Rks}, we can turn \eqref{eq:1} into the following Lippmann-Schwinger equation, and when $k$ is large enough, it admits a unique solution:
\begin{equation} \label{eq:LiSc}
	(I - \mathcal R^s_{k^{2s}+\mathrm{i}0} V) u^{\mathrm{sc}}
	= \mathcal R^s_{k^{2s}+\mathrm{i}0} (V u^{\mathrm{in}}).
\end{equation}

Regarded as an integral operator, the outgoing limiting resolvent $\mathcal R^s_{k^{2s}+\mathrm{i}0}$ has a kernel which is a fundamental solution of $((-\Delta)^s - k^{2s}) u = 0$ in $\R^d$, called the outgoing fundamental solution.
Denoted by $\Phi_s(x)$ the outgoing fundamental solution, the scattered field $u^{\mathrm{sc}}$ and the total field $u$ can be linked through 
\begin{equation} \label{1.5}
u^{\mathrm{sc}}(x)
= \int_{\mathbb{R}^d} \Phi_s(x-y) V(y) u(y) \dif y,
\end{equation}
which is equivalent to \eqref{eq:LiSc}.
For the case $0 < s < 1$, the asymptotics of the fundamental solution $\Phi_s(x)$ as $|x| \to +\infty$ in dimensions $d=1,2,3$ have been obtained in the recent work of Zilberberg, Cakoni and Vogelius \cite{Zilberberg2026}.

Combining \eqref{eq:1} with the outgoing limiting resolvent prescription, we obtain the fractional Schr\"odinger scattering problem considered in this paper.
The well-posedness of the forward problem is summarized in our first result.
As in the classical case, the scattered wave $u^{\mathrm{sc}}$ enjoys a decaying rate and the notion of the scattering amplitude $u^\infty$ can be defined from the asymptotics of $u^{\mathrm{sc}}$.

\begin{thm} \label{thm:1}
Assume $d \geq 3$, $\frac d {d+1} < s < \frac{d}{2}$, and $V \in L^\infty(\R^d)$ with compact support in domain $\Omega \subset \R^d$.
Let us fix some $1<p<q<\infty$ satisfying
\begin{equation} \label{les_con}
	\frac 2 {d+1} \leq \frac 1 p - \frac 1 q \leq \frac {2s} d, \qquad
	\frac 1 p > \frac {d+1} {2d}, \qquad
	\frac 1 q < \frac {d-1} {2d}.
\end{equation} 
Suppose that $k>k_0>0$ for certain $k_0$ depending on $p$, $q$, $\norm{V}_{L^\infty(\R^d)}$ and $\Omega$, and that $k^{2s}$ is not an eigenvalue of $(-\Delta)^s - V$ in $\mathbb R^d$.
Let $u = u^{\mathrm{in}} + u^{\mathrm{sc}}$ be the outgoing resolvent solution of \eqref{eq:1}, constructed by the limiting resolvent. The corresponding scattered part $u^{\mathrm{sc}}$ is called the outgoing scattered field.
Moreover, we have the following estimate uniform for $k \in (k_0,+\infty)$,
\begin{equation} \label{uni_mt}
	\|u^{\mathrm{sc}}\|_{L^{q}(\mathbb{R}^{d})}
	\leq C_{p,q}(k_0)\,k^{d(\frac{1}{p}-\frac{1}{q})-2s}\|V\|_{L^{\infty}(\R^d)}.
\end{equation}
Furthermore, 
\begin{equation}\begin{aligned}
		u^{\mathrm{sc}}
		\sim \frac {k^{2(1-s)}} s k^{\frac{d-3} 2} \frac {e^{-\frac{\mathsf{i}\pi(d-3)}{4}}} {2^{\frac {d+1} 2} \pi^{\frac {d-1} 2}} \frac{e^{ \mathsf{i}k|x|}}{|x|^{\frac {d-1} 2}}\, u^\infty( k, \hat x,\theta),
\end{aligned}\end{equation}
and $u^\infty(k, \hat x, \theta)$ is the scattering amplitude satisfying
\begin{equation}
	u^\infty(k, \hat x, \theta )
	= \int_{\mathbb{R}^d} e^{-\mathsf{i}k \hat x \cdot y} V(y) \big( e^{\mathrm{i}k y\cdot \theta} + u^{\mathrm{sc}}(y) \big) \,\mathrm{d}y. 
\end{equation}
\end{thm}

Note that the fractional Laplacian is defined using Fourier transform, and the resolvent estimates are based on $L^p$ for $p \neq 2$.
We give a proper definition of $(-\Delta)^s$ in $\mathcal S'(\R^d)$ using the Fourier transform which is based on $L^2$.
These $L^p(\R^d)$ spaces are subspaces of $\mathcal S'(\R^d)$, so these arguments are properly reasoned.

The estimate \eqref{uni_mt} in Theorem \ref{thm:1} can be regarded as a generalization of the pioneering work \cite{KRS} by Kenig, Ruiz and Sogge that established the strong $(p,q)$-type bound of the classical resolvent.
Later on, Guti\'errez in \cite{GuSu} gave a decaying rate with respect to the frequency parameter of the resolvent.

It is important to note that the scattering amplitude depends on three variables, namely the triplet $(k, \hat{x}, \theta)\in \mathbb{R}_{>0}\times \mathbb{S}^{d-1}\times\mathbb{S}^{d-1}$. The question is: Is it necessary to vary all three factors in order to recover the potential function, or is it sufficient to vary only some of them? This gives rise to several different variants of the problem in the local case $s=1$.  These problems can be similarly formulated in the fractional case as well. We now mention a particular variant of the problem in the local case. The  problem corresponds to the case where all three variables — $k$, $\hat{x}$, and $\theta$ — are varied. That is, the goal is to recover the potential from knowledge of the amplitude for all 
\[
(k, \hat{x}, \theta) \in \mathbb{R}_{>0} \times \mathbb{S}^{d-1} \times \mathbb{S}^{d-1},
\]
or at least for 
\[
k \geq k_0 \quad \text{with} \quad k_0 \in \mathbb{R}_{>0}, \quad \text{and} \quad (\hat{x}, \theta) \in \mathbb{S}^{d-1} \times \mathbb{S}^{d-1}.
\]  
This problem is also known as the Born approximation problem and can be related to the Borg-Levinson problem  for $s=1$.
The problem concerns the recovery of the potential from the spectral data set and the Neumann boundary values of the eigenfunctions, that is, the normal derivatives of the eigenfunctions on the boundary.
In the local case, the authors in \cite{nachman1992inverse} transformed the Borg--Levinson problem into a scattering problem and showed that the scattering amplitude uniquely determines the potential.
Since the Neumann derivatives of the eigenfunctions determine the \emph{Dirichlet-to-Neumann} map associated with the Dirichlet operator involving the potential, they also determine the Neumann boundary values of the solution corresponding to a prescribed Dirichlet boundary condition.
The eigenvalues, together with the Dirichlet and Neumann boundary data, constitute a data set of dimension \(1 + (d - 1) + (d - 1)\), from which the potential can be uniquely determined.
For the local case, this is analogous to the scattering case, where the scattering amplitude is known for \((k, \hat{x}, \theta) \in \mathbb{R}_{>0} \times \mathbb{S}^{d-1} \times \mathbb{S}^{d-1}\), that is, a data set of dimension \(1 + (d - 1) + (d - 1)\), from which the potential can be uniquely determined.
For a comprehensive study of the Borg-Levinson problem we refer the readers to \cites{borg1946umkehrung, levinson1949inverse, nachman1988n, novikov1988multidimensional, soccorsi:hal-03571903, gel1951determination}.

The fractional Borg–Levinson problem can be formulated in a similar manner, where the \emph{Dirichlet-to-Neumann} map is well defined, as established in \cites{grubb2015fractional, grubb2018green}. In \cite{das2025fractional}, the authors determined the potential uniquely from the spectral data using a density argument and the boundary control techniques as in \cite{biccari2025boundary}.
Here also the dimension of the known data set is \(1+(d-1)+(d-1)\) which is exactly the same as the dimension of the known data set for the scattering problem. In this article, analogous to the local case, we aim to build a bridge between the inverse scattering problem and the fractional Borg–Levinson problem; specifically, we investigate whether the scattering amplitude uniquely determines the potential. Although our results require a specific lower bound on the fractional exponent, we believe that the result holds for all fractional exponents strictly greater than \(\frac{1}{2}\), as is the case in the fractional Borg--Levinson problem. The main result in our article is the following:

\begin{thm} \label{thm:2}
Assume $d \geq 3$, $\frac d {d+1} < s < \frac{d}{2}$, and $V_j \in L^\infty(\R^d)$ with compact support $\Omega \subset \mathbb{R}^d$ for $j=1,2$.
Suppose that $k>k_0>0$ and that $k^{2s}$ is not in the point spectrum of $(-\Delta)^s - V_1$ and $(-\Delta)^s - V_2$ in $\mathbb R^d$.
Let $u_j~(j=1,2)$ be outgoing resolvent solutions of \eqref{eq:1} corresponding to potential $V_j$, respectively.
Suppose for all $(k, \hat x, \theta) \in (k_0,+\infty) \times \mathbb{S}^{d-1} \times \mathbb{S}^{d-1}$, there hold
\begin{equation}
	u^{\infty}_1(k, \hat x, \theta)= u^{\infty}_2(k, \hat x, \theta). 
\end{equation}
Then $V_1=V_2$.
\end{thm}

\begin{rem}
	The restriction $d\ge 3$ comes from the method of proof. In particular, we use the $L^p$-$L^q$ resolvent estimates for
	the fractional Helmholtz operator in the form available in \cite{HYZ}, which are stated for dimensions $d\ge 3$. Our
	choice of admissible exponents and the decay argument in the uniqueness proof are based on this framework. We expect
	that an analogue in dimension $d=2$ may also hold, but it would require a separate analysis of the resolvent estimates
	and the far-field asymptotics, which is beyond the scope of the present paper.
\end{rem}

\begin{rem}
	The range $\frac{d}{d+1}<s<\frac d2$ is technical and comes from the $L^p$-$L^q$ resolvent estimates used in the
	proof. The lower bound is essential for our argument, since we need admissible exponents $p,q$ such that
	\[
	\frac d2\Big(\frac1p-\frac1q\Big)-s<0,
	\]
	while for $0<s<\frac{d}{d+1}$ the required uniform resolvent estimates fail; see \cite{HYZ}. We expect the uniqueness
	result to remain true in a larger range, possibly for all $s>\frac12$, but this would require different techniques.
\end{rem}

The potential recovery using far-field data in classical case has been studied, see e.g.~\cite{liu2015determining}.
In the fractional case, most of the work is about the fractional Calder\'on problem with boundary measurement, for example, \cites{GSU20, ghosh2021non, GRSU20, bhattacharyya2021inverse, GU2021calder} and references therein.
In this work, we study the inverse scattering problem, with the measurement being the far-field pattern.
Furthermore, in the system we study, there is a frequency parameter $k^{2s}$, which makes the
behavior of the solution different from that of the fractional Calder\'on problem.
As we finalize the paper, we notice the recent work \cite{uhlmann2025recovering} that deals with the $s = 1/2$ case with scattering matrix as their data.
They can recover the potential in the radial case up to some smooth error.
Compared to their work, our approach relies on a $(p,q)$-type estimate for the fractional resolvent, and our unique recovery result is precise.

\subsection{Inverse scattering and Heisenberg-\texorpdfstring{$\mathcal{S}$}{S} matrix}

The Helmholtz operator can be viewed as the Laplace transform of the wave operator. Let us denote the full Hamiltonian by
\[
H_{V}= H_{0}+V, \ \ \mbox{where} \ H_0=\frac{1}{2}\Delta.
\]
Due to the asymptotic completeness of the Schr\"odinger  operator for the local case \(s=1\), i.e., $i\partial_t+\Delta$, the following wave operators are well defined
\[
W^{\pm}= \lim\limits_{t\to \pm \infty} e^{itH_{V}} e^{-itH_{0}}.
\]
Thus one can define the scattering operator (known to be Heisenberg-$\mathcal{S}$ matrix) $\mathcal{S}:L^2(\mathbb{R}\times\mathbb{S}^{d-1})\to L^2(\mathbb{R}\times\mathbb{S}^{d-1})$ as:
\[
\mathcal{S}= \left(W^+\right)^* W^{-}.
\]
Through Lax--Phillips scattering theory, one can show that the operator \( \mathcal{S} - \mathcal{I} \) is compact and can be represented by a scattering kernel via a specific Radon transform. Here $\mathcal{I}$ is the identity operator. Moreover, this kernel is further related to the scattering amplitude through a Fourier-type transform. The precise relations are given by
(\cites{GE, uhlmann2000inverse})
\begin{equation} \label{matrix-amplitude}
\left \{
\begin{aligned}
	(\mathcal{S}-\mathcal{I}) f(s,\hat{x}) & = \int_{\mathbb{R}} \int_{\mathbb{S}^{d-1}} k (s-s',\hat{x},\theta) f(s',\theta)\, \rm{d} \theta \, \rm{d} s' \\
	u^{\infty} (k,\hat{x},\theta) & = c_d k^{d-3} \int_{\mathbb{R}} e^{ik\lambda} k (\lambda,\hat{x}, \theta)\, \rm{d} s,
\end{aligned}
\right.
\end{equation}
where $k(s,\hat{x},\theta)$ is the scattering kernel and $c_d$ is a constant. Therefore, in the local case, the equality of the Heisenberg \( \mathcal{S} \)-matrices is equivalent to the equality of the scattering kernels, which in turn is equivalent to the equality of the scattering amplitudes. Therefore, determining the potential via the Heisenberg-\( \mathcal{S} \) matrix is equivalent to determining it from the scattering amplitude. For a comprehensive study we refer the readers to \cites{ER95,GE,AgSh}. For the local case, an analogous version of the problem in the Riemannian manifold setting can be found in \cite{uhlmann2000inverse}, where the author recovers the geometry of the manifold and establishes the equivalence between studying the Heisenberg-\( \mathcal{S} \) matrix and the scattering amplitude for determining the manifold’s geometry. 

The same problem can be extended to the fractional settings. Here also the wave operators are well defined:
\[
W_{s}^{\pm}= s-\lim\limits_{t\to\pm \infty} e^{itH_V} e^{-itH_{0}}, \ \ \mbox{where} \ H_{0}=\frac{1}{2s} (-\Delta)^{2s} \ \mbox{and} \ H_V=H_{0}+V.
\]
For asymptotic completeness and well definedness of the wave operators, we refer the readers to \cites{kitada2010scattering, kitada2011remark}. Thus we can define the scattering operator $\mathcal{S}_{s}: L^2(\mathbb{R}\times\mathbb{S}^{d-1})\to L^2(\mathbb{R}\times\mathbb{S}^{d-1})$ as follows
\[
\mathcal{S}_{s}= \left(W_s^+\right)^{*} W_s^{-}.
\]
In \cite{MR4058699}, the author established that the Heisenberg-$\mathcal{S}$ matrix or the scattering operator uniquely determines the potential when the exponent \( s \in \left(\frac 12, 1\right)\). This result was later extended to the case \( s = \frac{1}{2} \) in \cite{uhlmann2025recovering}. In this article, we determine the potential from the scattering amplitude. Unlike the local case \( s = 1 \), the equivalence between studying the Heisenberg-\( \mathcal{S} \) matrix and the scattering amplitude (the relation \eqref{matrix-amplitude} for the local case) is not yet known for the case of a fractional exponent.

\subsection{Some further comments and open problems}

In this part, we would like to mention some other different variants of the scattering problems. Is it possible to determine the unknown potential uniquely from strictly less than $2d-1$ dimensional data? We mention some of the problems here. All the problems are still unresolved in the fractional exponent case.
\begin{itemize}

\item \textbf{Fixed frequency problem (\(2d-2\) dimensional data):} \\
Let $k>0$ be fixed. Is it possible to determine the potential uniquely from the scattering amplitude $u^{\infty}$ for all \((\hat{x},\theta) \in \mathbb{S}^{d-1} \times \mathbb{S}^{d-1}\)? \\
This problem is solved for the local case \(s=1\) in \cites{GE, lax2014functional}.
The techniques used for the local case heavily depend on the Rellich lemma and analyticity of the polynomial \(\xi^{2s}-k^{2s}\), when $s=1$. These techniques fail when $s$ is a fractional quantity. We discuss this in detail in the appendix of this article.

\item \textbf{Fixed angle problem ($d$ dimensional data):} \\
Let $\hat{x}$ be fixed. Is it possible to determine the potential uniquely from the scattering amplitude $u^{\infty}$ for all \((k,\theta)\in \mathbb{R}_{>0}\times\mathbb{S}^{d-1}\)?
Simple algebraic computation leads to the following identity
\[
u^{\infty} (k,\hat{x},\theta)=u^{\infty} (k,-\hat{x},-\theta).
\]
Hence, fixing $\hat{x}\in\mathbb{S}^{d-1}$ is the same as fixing $\theta\in\mathbb{S}^{d-1}$, and the latter is known as the single measurement problem.

\item  \textbf{Back-scattering problem ($d$ dimensional data):}\\
Is it possible to determine the potential uniquely from the scattering amplitude $u^{\infty}(k,-\theta,\theta)$ for all $(k,\theta)\in\mathbb{R}_{>0}\times\mathbb{S}^{d-1}$?
\end{itemize}
Even for the local case the last two problems are open in general. The one dimensional version for the local case can be found in \cites{deift1979inverse, marchenko2011sturm}.  For various results and comprehensive study we refer the readers to 
\cites{uhlmann2001time, barcelo2020uniqueness, eskin1992inverse, uhlmann2014uniqueness, Rakesh_2020, MR3372472, MR4170189}.

\section{Resolvent estimate and scattering amplitude}

\subsection{Domain of the fractional Laplacian} \label{subsec:Ds}

Throughout the work we use the following convention for the Fourier transforms:
\[
\mathcal F f(\xi) := (2\pi)^{-\frac d 2} \int_{\R^d} e^{-\mathsf{i}x\cdot \xi} f(x) \dif x, \
\mathcal F^{-1} f(x) := (2\pi)^{-\frac d 2} \int_{\R^d} e^{\mathsf{i}x\cdot \xi} f(x) \dif x,
\]
We may denote $\mathcal F f$ as $\hat f$ for short.
In this work, the fractional Laplacian $(-\Delta)^s$ is defined by Fourier transform as $\mathcal F^{-1} |\xi|^{2s} \mathcal F$,
\begin{equation*}
(-\Delta)^s f := \mathcal F^{-1} |\xi|^{2s} \mathcal F f.
\end{equation*}
Due to the non-smoothness of $|\xi|^{2s}$ at $\xi = 0$, $(-\Delta)^s$ does not map the Schwartz space into itself.
To see that, we choose $\varphi \in \mathcal S(\R^d)$ and compute $x^\alpha D^\beta (-\Delta)^s \varphi$:
\begin{align}
x^\alpha D_x^\beta (-\Delta)^s \varphi(x)
& \simeq x^\alpha D_x^\beta \int_{\R^d} e^{\mathsf{i}x\cdot \xi} |\xi|^{2s} \hat \varphi(\xi) \dif \xi
= x^\alpha \int_{\R^d} e^{\mathsf{i}x\cdot \xi} \xi^\beta |\xi|^{2s} \hat \varphi(\xi) \dif \xi \nonumber \\
& = \int_{\R^d} D_\xi^\alpha (e^{\mathsf{i}x\cdot \xi}) \xi^\beta |\xi|^{2s} \hat \varphi(\xi) \dif \xi. \label{eq:absv}
\end{align}
For the last integration, it is not guaranteed that the integration by parts can always be performed for every multi-index $\alpha$ and $\beta$, thus $(-\Delta)^s \varphi(x)$ might not be decaying faster than polynomials, so the claim follows.
As a byproduct, we can conclude from the calculation above that
\[
(-\Delta)^s \colon \mathcal S(\R^d) \to C^\infty(\R^d).
\]

However, the dual space of $C^\infty(\R^d)$, which is the set of distributions with compact support, is not big enough to have useful distributions contained in it, such as the plane wave $e^{\mathsf{i}kx\cdot \theta}$.
To fix this, we shrink the domain to the following subset of $\mathcal S(\R^d)$, called {\it restricted Schwartz space},
\begin{equation*}
\mathcal S_0(\R^d)
:= \big\{ \varphi \in \mathcal S(\R^d) \,;\, \partial^\alpha \hat \varphi(0) = 0 \ \text{for every multi-index}~ \alpha \big\}.
\end{equation*}
The following result shows the restricted Schwartz space is small enough such that it is a good test function space for the fractional Laplacian $(-\Delta)^s$, but is not too small.

\begin{prop}
$\mathcal S_0(\R^d)$ is dense in $L^2(\R^d)$.
Moreover, the map
\begin{equation} \label{eq:DsD}
	(-\Delta)^s \colon \varphi \in \mathcal S_0(\R^d) \ \mapsto \ \mathcal F^{-1} |\xi|^{2s} \mathcal F \varphi \in \mathcal S_0(\R^d).
\end{equation}
is well-defined in $\mathcal S_0(\R^d)$, and maps $\mathcal S_0(\R^d)$ into itself.
\end{prop}

\begin{proof}
Fix a sequence of cutoff functions $\chi_j~(j \geq 1)$ that satisfy the following,
\begin{equation*}
	\left\{\begin{aligned}
		& \chi_j(\xi) = 0, && |\xi| \in [0, 2^{-j-1}] \cup [2^{j+1},+\infty), \\
		& \chi_j(\xi) = 1, && |\xi| \in [2^{-j}, 2^j], \\
		& 0 \leq \chi_j(\xi) \leq 1, && \text{otherwise}.
	\end{aligned}\right.
\end{equation*}
For any $f \in L^2(\R^d)$, we define
\begin{equation} \label{eq:fj}
	f_j := \mathcal F^{-1} \chi_j \mathcal F(\chi_j f).
\end{equation}
First, we need to show $f_j$'s are contained in the restricted Schwartz space. To see that, we just need to check $f_j \in \mathcal S(\R^d)$, as the conditions $\partial^\alpha \hat f_j(0) = 0$ are automatically satisfied due to the cutoff $\chi_j$ next to the $\mathcal F^{-1}$ in \eqref{eq:fj}.
$f_j \in \mathcal S(\R^d)$ is equivalent to $\hat f_j \in \mathcal S(\R^d)$, and $\hat f_j = \chi_j \mathcal F(\chi_j f)$, so $\hat f_j$ is compactly supported, thus it is left to check that $\mathcal F(\chi_j f)$ is smooth.
This is also true, because the Fourier transform of a compactly supported $L^2$ function is smooth.
Therefore, we can conclude that $\{f_j\}_{j \geq 1} \subset \mathcal S_0(\R^d)$.

Second, we check $f_j \to f$ in $L^2(\R^d)$.
By the Plancherel theorem we have
\begin{align*}
	\norm{f_j - f}_{L^2}
	& = \norm{\chi_j \mathcal F(\chi_j f) - \hat f}_{L^2} \\
	& \leq \norm{(1 - \chi_j) \mathcal F(\chi_j f)}_{L^2} + \norm{\mathcal F ((1-\chi_j) f)}_{L^2} \\
	& \leq \norm{(1 - \chi_j) \mathcal F((1 - \chi_j) f)}_{L^2} + \norm{(1 - \chi_j) \hat f}_{L^2} + \norm{(1-\chi_j) f}_{L^2}.
\end{align*}
For the first term, because $0 \leq 1 - \chi_j \leq 1$ and the Fourier transform $\mathcal F$ is unitary in $L^2(\R^d)$, so
\[
\norm{(1 - \chi_j) \mathcal F((1 - \chi_j) f)}_{L^2}
\leq \norm{\mathcal F((1 - \chi_j) f)}_{L^2}
\leq \norm{(1 - \chi_j) f}_{L^2},
\]
thus
\begin{equation} \label{eq:fjf}
	\norm{f_j - f}_{L^2}
	\leq 2\norm{(1 - \chi_j) f}_{L^2} + \norm{(1 - \chi_j) \hat f}_{L^2}.
\end{equation}
These two terms converge to zero as $j \to +\infty$ by the dominated convergence theorem.
Hence, we proved $\mathcal S_0(\R^d)$ is dense in $L^2(\R^d)$.

To show the second conclusion, we fix $\varphi \in \mathcal S_0(\R^d)$.
As indicated by \eqref{eq:absv},
\begin{equation*}
	x^\alpha D_x^\beta (-\Delta)^s \varphi(x)
	\simeq \int_{\R^d} D_\xi^\alpha (e^{\mathsf{i}x\cdot \xi}) \xi^\beta |\xi|^{2s} \hat \varphi(\xi) \dif \xi.
\end{equation*}
Because $\partial^\gamma \hat \varphi(0) = 0$ for any $\gamma$, we see $\psi(\xi) := \xi^\beta |\xi|^{2s} \hat \varphi(\xi)$ is a Schwartz function, so we can continue the calculation,
\begin{equation*}
	x^\alpha D_x^\beta (-\Delta)^s \varphi(x)
	\simeq \int_{\R^d} e^{\mathsf{i}x\cdot \xi} D_\xi^\alpha \psi(\xi) \dif \xi.
\end{equation*}
This integral is bounded, so $(-\Delta)^s \varphi$ is a Schwartz function.
Moreover, the condition $\partial^\alpha \hat \varphi(0) = 0$ implies for any positive integer $k$, there exists $\phi_k$ such that $\hat \varphi(\xi) = |\xi|^{2k} \phi_k(\xi)$, so
\(
\partial^\alpha \mathcal F \{ (-\Delta)^s \varphi \} = \partial^\alpha \big( |\xi|^{2(s+|\alpha|)} \phi_{|\alpha|}(\xi) \big),
\)
thus
\[
\partial^\alpha \mathcal F \{ (-\Delta)^s \varphi \}(0) = 0.
\]
Hence, $(-\Delta)^s \varphi \in \mathcal S_0(\R^d)$.
The proof is complete.
\end{proof}

The definition \eqref{eq:DsD} can be extended to the set of distributions by density arguments.
Because $\mathcal S_0(\R^d)$ is a subspace of $\mathcal S(\R^d)$, the dual space of $\mathcal S_0(\R^d)$ contains $\mathcal S'(\R^d)$, so the definition \eqref{eq:DsD} can be extended to $\mathcal S'(\R^d)$ by duality arguments:
\begin{equation} \label{eq:DsD2}
((-\Delta)^s u, \varphi) := (u, (-\Delta)^s \varphi), \quad \text{for} \quad u \in \mathcal S'(\R^d), \ \ \varphi \in \mathcal S_0(\R^d).
\end{equation}
The distribution space $\mathcal S'(\R^d)$ is a subset of the dual space of $\mathcal S_0(\R^d)$, but extending the domain of definition of $(-\Delta)^s$ from $\mathcal S_0(\R^d)$ to $\mathcal S'(\R^d)$ is already enough for our purpose.

\subsection{Fundamental solution}

Let us investigate the fundamental solutions since they are crucial to our analysis.
Let $\Phi_{\pm,s}(x)$ be defined by
\begin{equation}\label{phi}
\Phi_{\pm,\,s}(x) := \lim_{\varepsilon \to 0^+}\Phi_{\pm,\, s,\, \varepsilon}(x),
\end{equation}
where $\Phi_{\pm,\,s,\, \varepsilon}$ are given by
\begin{equation}
\Phi_{\pm,\,s,\, \varepsilon}(x) = (2\pi)^{-d} \int\displaylimits_{\mathbb{R}^d} {\frac{e^{\mathsf{i}x\cdot \xi}}{|\xi|^{2s} - (k\pm \mathsf{i}\varepsilon)^{2s}}\,\mathrm{d}\xi}.
\end{equation}
$\Phi_{\pm,\,s,\, \varepsilon}$ is a fundamental solution of
\begin{equation}
\left[(-\Delta)^s - (k\pm \mathsf{i}\varepsilon)^{2s}\right] \Phi_{\pm,\,s,\, \varepsilon}(x) = \delta_0(x) \quad\text{ for }x\in  \mathbb{R}^d.
\end{equation}
Here $\delta_0$ signifies the Dirac delta measure located at the origin.
To deal with the limit of $(k \pm \mathsf{i}\varepsilon)^{2s}$, we need the following lemma.

\begin{lem} \label{lem:eps1}
If $s>0$ and $k>0$, it holds that
\begin{equation}\label{eq lim}
	\lim_{\varepsilon \to 0^+}(k \pm \mathsf{i}\varepsilon)^{2s} = 
	\lim_{\varepsilon \to 0^+}(k^{2s} \pm \mathsf{i}\varepsilon).
\end{equation}
\end{lem}

\begin{proof}
When $\varepsilon>0$, there exist functions $\varrho$ and $\theta$ of $k$ and $\varepsilon$ such that
\begin{equation*}
	(k \pm \mathsf{i}\varepsilon)^{2s} = \varrho^{2s} e^{\pm \mathsf{i} 2s \theta}, \quad \lim_{\varepsilon \to 0^+}\varrho(k,\varepsilon) = k, \ \text{ and }\  \lim_{\varepsilon \to 0^+}\theta(k,\varepsilon)=0.
\end{equation*}
Hence, $(k \pm \mathsf{i}\varepsilon)^{2s} = \varrho^{2s} \LC  \cos(2s\theta) \pm \mathsf{i} \sin (2s\theta) \RC $, and therefore,
$$
\lim_{\varepsilon \to 0^+}(k \pm \mathsf{i}\varepsilon)^{2s} = \lim_{\varepsilon \to 0^+} \varrho^{2s} \left( \cos(2s\theta) \pm \mathsf{i} \sin (2s\theta)\right) =\lim_{\varepsilon \to 0^+}(k^{2s} \pm \mathsf{i}\varepsilon)
$$
as desired.	This completes the proof.
\end{proof}

With the help of Lemma \ref{lem:eps1}, we can compute
\begin{align}
\Phi_{\pm,\,s}(x) 
& = \lim_{\varepsilon \to 0^+} (2\pi)^{-d} \int\displaylimits_{\mathbb{R}^d} {\frac {e^{\mathsf{i}x\cdot \xi}} {|\xi|^{2s} -k ^{2s}\mp \mathsf{i}\varepsilon}\,\mathrm{d}\xi} \nonumber \\
&= (2\pi)^{-d} \mathrm{P.V.}\int\displaylimits_{\mathbb{R}^d} {e^{\mathsf{i}x\cdot \xi} \left[  \frac{1}{|\xi|^{2s} -k ^{2s}}\pm \mathsf{i}\pi \delta \LC |\xi|^{2s} -k ^{2s}\RC \right] \mathrm{d}\xi} \nonumber \\
&= (2\pi)^{-d} \mathrm{P.V.} \int\displaylimits_{\mathbb{R}^d} { \frac {e^{\mathsf{i}x\cdot \xi}} {|\xi|^{2s} -k ^{2s}}\,\mathrm{d}\xi} \pm \mathsf{i}\pi (2\pi)^{-d} \int\displaylimits_{|\xi|=k}{\frac{e^{\mathsf{i}x\cdot \xi}}{2s k^{2s -1}}\,\mathrm{d}S(\xi)} \nonumber \\
&= \frac{k^{d-2s}}{(2\pi)^d} \bigg[ \mathrm{P.V.}  \int\displaylimits_{\mathbb{R}^d}{\frac{e^{\mathsf{i}kx\cdot \xi}} {|\xi|^{2s}-1} \mathrm{d}\xi} \pm \frac{\mathsf{i}\pi}{2s}\int\displaylimits_{\mathbb{S}^{d-1}}{e^{\mathsf{i}kx\cdot\omega}\,\mathrm{d}S(\omega)} \bigg], \label{2}
\end{align}
where $\mathrm{P.V.}$ stands for the Cauchy principal value.
Here in the second last equality we used the action \[\delta	(C)(\varphi) = \int\displaylimits_{C(x)=0}
{\frac{\varphi(x)}{|\nabla C(x)|}\,\mathrm{d}S(x)}.\]
Let us denote
\begin{equation}
\mathcal{S}_d(t):= (2\pi)^{-d/2} \int\displaylimits_{\mathbb{S}^{d-1}} {e^{\mathsf{i}te_1\cdot \omega}\,\mathrm{d}S(\omega)},
\end{equation}
where $e_1$ stands for the unit vector with its first component being $1$ and rest $0$, then it can be seen that $\mathcal{S}_d(t)$ is continuous on $[0,\infty)$ and
\begin{equation}
\label{4}
\mathcal{S}_d(t) =
\begin{cases}
	\frac{J_{d/2-1}(t)}{t^{d/2-1}}, &t\neq 0,\\[1mm]
	\frac{2^{-d/2+1}}{\Gamma(d/2)}, & t=0,
\end{cases}
\quad d\ge 2, \quad \text{ and } \quad \mathcal{S}_3(t) = \sqrt{\frac{2}{\pi}}\frac{\sin t}{t}
\end{equation}
where $J_\alpha$ is the usual Bessel function of the first kind.

\medskip
We continue from \eqref{2} by using the functions $\mathcal{S}_d(t)$ to write further
\begin{equation} \label{3s}
	\Phi_{\pm,\, s}(x)
	= \frac{k^{d-2s}}{(2\pi)^d} \bigg[ \mathrm{P.V.}  \int\displaylimits_{0}^\infty{\frac{t^{d-1}\mathcal{S}_d(k|x|t)}{t^{2s}-1}\, \mathrm{d}t} \pm \frac{\mathsf{i}\pi}{2s} \;\mathcal{S}_d(k|x|) \bigg], \ \ x\in\mathbb{R}^d.
\end{equation}
For $s=1$, we obtain the expression of the fundamental solution for the Helmholtz operator $(-\Delta) - k^{2}$ in $\mathbb{R}^d$ as
\begin{equation}
\label{31}
\Phi_{\pm,\, 1}(x) =
\frac{k^{d-2}}{(2\pi)^d} \bigg[ \mathrm{P.V.}  \int\displaylimits_{0}^\infty{\frac{t^{d-1}\mathcal{S}_d(k|x|t)}{t^{2}-1}\, \mathrm{d}t} \pm \frac{\mathsf{i}\pi}{2} \;\mathcal{S}_d(k|x|) \bigg], \ \ x\in\mathbb{R}^d.
\end{equation}
For any $d\geq 2$, let us mention the standard asymptotic behavior of $\Phi_{\pm, 1}$, this can be found for instance in \cite{GE}*{Lemma 19.3}:
\begin{equation}\label{31asy}
\Phi_{\pm,\, 1}(x)= k^{\frac{d-3}{2}}\frac{e^{-\frac{\mathsf{i}\pi(d-3)}{4}}}{2^{\frac{d+1}{2}}\,\pi^{\frac{d-1}{2}}}\,\,\frac{e^{\pm \mathsf{i}k|x|}}{|x|^{\frac{d-1}{2}}} +\mathcal{O}\left(\frac{1}{|x|^{\frac{d+1}{2}}}\right), \quad \text{ as } |x| \to \infty.
\end{equation}
In particular for $d=3$, we have
\begin{equation} \label{313}
\Phi_{\pm,\, 1}(x)= \frac{1}{4\pi}\,\frac{e^{\pm \mathsf{i}k|x|}}{|x|}, \quad x\in\mathbb{R}^d.
\end{equation}

\subsection{Analysis via resolvent operator}

Let us mention the following expression of the resolvent of the fractional power of the negative Laplacian, as established in  \cite{M}*{Proposition 5.3.3}:
\begin{align}\label{decomposition}
\begin{split}
	\LC (-\Delta)^{s}-k^{2s}\RC ^{-1}&=\frac{k^{2(1-s)}}{s}(-\Delta-k^2)^{-1}  \\
	&\quad+\frac{\sin{s\pi}}{\pi}\int_{0}^{\infty}\frac{\lambda^{s}(\lambda-\Delta)^{-1}}
	{\lambda^{2s}-2\lambda^{s}k^{2s}\cos{s\pi}+k^{4s}}\,  \mathrm{d}\lambda.
\end{split}
\end{align}
In other words,
\begin{align}\label{resols1}
\Phi_{\pm,\, s}=\frac{k^{2(1-s)}}{s}\Phi_{\pm,\, 1} +\frac{\sin{s\pi}}{\pi}\int_{0}^{\infty}\frac{\lambda^{s}(\lambda-\Delta)^{-1}\delta_y(x)}
{\lambda^{2s}-2\lambda^{s}k^{2s}\cos{s\pi}+k^{4s}}\,  \mathrm{d}\lambda,
\end{align}
where in the second term of the above expression we have (see \cite{S}*{page 132})
\begin{equation}\label{lD}
(\lambda-\Delta)^{-1}\delta_y(x)=c\int_{0}^{\infty}e^{-\delta\lambda-\frac{|x-y|^2}{4\pi\delta}}\delta^{\frac{-d+2}{2}}\, \frac{\mathrm{d}\delta}{\delta},\quad \lambda>0.
\end{equation}
Relying on the decomposition \eqref{decomposition}, and using Stein’s oscillatory integral techniques (see \cite{KRS}), Huang, Yao and Zheng established the uniform estimate for fractional resolvents in \cite{HYZ}*{Theorem 1.4}.

\begin{prop}\label{resol}
Let $d\geq 3$.
If $\frac{d}{d+1}\leq s<\frac d2$ and  $1<p<q<\infty$ are Lebesgue exponents satisfying
\begin{equation} \label{les_con0}
	\frac1p-\frac1q=\frac{2s}{d} \quad \text{and} \quad \min\left(\ \Big|\frac1p-\frac12\Big|,\Big|\frac1q-\frac12\Big|\ \right) > \frac{1}{2d},
\end{equation}
then there is a uniform constant $C_{p,q}>0$, such that for all $z\in \mathbb{C}$,
\begin{equation}\label{hyz2}
	\|u\|_{L^q(\mathbb{R}^d)}\leq C_{p,q}\|((-\Delta)^{s}-z)u\|_{L^p(\mathbb{R}^d)},\quad u\in C_0^\infty(\mathbb{R}^d).
\end{equation}
On the other hand, if $0<s<\frac{d}{d+1}$, then no such uniform estimates exist.
\end{prop}
Note that if $z=0$, \eqref{hyz2} becomes the classical Hardy-Littlewood-Sobolev inequality which is true for  $0<2s<d$ and $\frac1p-\frac1q=\frac{2s}{d}$, $1<p<q<\infty$.

The exclusion of $0<s<\frac{d}{d+1}$ to have a uniform estimate like \eqref{hyz2} is consistent with the Stein-Tomas estimate; i.e.~if we assume such uniform estimate takes place in this regime it contradicts
\begin{equation*}
\|\hat f\|_{L^2(\mathbb S^{d-1})}
\leq C\|f\|_{L^p(\mathbb{R}^d)}, \quad 1\leq p\leq \frac{2(d+1)}{d+3}.
\end{equation*}

As an application of the above results, authors in \cite{HYZ}*{Corollary 3.2} further find the following $L^p$-$L^q$ estimates.

\begin{prop} \label{prop5}
Let $d\geq 3$, $\frac{d}{d+1}\leq s<\frac{d}{2}$ and  $1<p<q<\infty$  are Lebesgue exponents satisfying
\begin{equation} \label{eq:pq2}
	\frac 2 {d+1} \leq \frac 1 p - \frac 1 q \leq \frac {2s} {d}, \qquad
	\frac 1 p > \frac {d+1} {2d}, \qquad
	\frac 1 q < \frac {d-1} {2d},
\end{equation}
then there is a uniform constant $C>0$, such that
\begin{equation} \label{resol_frac_2}
	\sup_{0<\varepsilon<1} \|((-\Delta)^{s}-(\lambda+\mathrm{i}\varepsilon))^{-1} \|_{L^{p}-L^{q}} \leq C |\lambda|^{\frac{d}{2s}(\frac{1}{p}-\frac{1}{q})-1}, \quad \lambda > 0.
\end{equation}
\end{prop}

This result could be considered as a generalization of \cite{KRS}*{Theorem 2.1} and  \cite{GuSu}*{Theorem 6} to the fractional case.
Kenig, Ruiz and Sogge in \cite{KRS}*{Theorem 2.1} established the boundedness estimate for the local resolvent operator, 
that for $d\ge 3$, there is a constant $C_p>0$ independent of $\lambda$ such that
\begin{align} \label{exist_local_1}
	\sup_{0<\varepsilon<1}\|(-\Delta-(\lambda+i\varepsilon))^{-1}\|_{L^p-L^q}\leq C_p, ~~~\lambda>0,
\end{align}
for $p$, $q$ satisfying \eqref{les_con0} with $s = 1$.
Later, Guti\'{e}rrez in \cite{GuSu}*{Theorem 6} generalized this to
\begin{align}\label{exist_local_2}
	\sup_{0<\varepsilon<1} \| (-\Delta - (\lambda + i\varepsilon))^{-1} \|_{L^p-L^q}\leq C\lambda^{\frac d 2 (\frac 1 p - \frac 1 q) - 1}, \quad \lambda > 0,
\end{align}
for $p$, $q$ satisfying \eqref{eq:pq2} with $s = 1$.
Taking $\varepsilon\rightarrow 0^{+}$ in the above estimates, one obtains the boundedness estimate for the resolvent 
$\mathcal{R}_{\lambda}=(-\Delta-\lambda)^{-1}$.

\subsection{Sommerfeld radiation condition}

In this paper, the outgoing solution is defined via the limiting resolvent.
However, as in the local case, the outgoing solution for the fractional problem can also be characterized as the unique solution to the equation when a suitable radiation condition is imposed.
To explain, let us first revise the Helmholtz equation with $s = 1$,
\begin{equation}\label{local_Helmhotz}
	(-\Delta - k^2) u = f  \ \ \mbox{in} \ \ \mathbb{R}^{d}.
\end{equation}
Since $k^2 > 0$ is contained in the essential spectrum of $-\Delta$, the uniqueness of solutions for the Helmholtz equation becomes complicated.
Since different solutions vanish at infinity, it is necessary to introduce additional conditions at infinity to determine the unique solution, that is the so-called Sommerfeld (outgoing) radiation condition:
\begin{equation}
	u(|x|)=O(|x|^{\frac{1-d}{2}}),\quad\frac{\partial u}{\partial |x|}-\mathrm{i}k u=o(|x|^{\frac{1-d}{2}}),
\end{equation}
see \cite{ReFr}, or equivalently (see \cite{GuSu}*{Corollary 1})
\begin{equation}\label{uni2}
	\mathop{\mathrm{lim}}\limits_{R\longrightarrow+\infty}\frac{1}{R}\int_{1<|x|<R}|\nabla u-\mathrm{i}ku\,\hat x|^{2}dx=0, \quad\mbox{where }\hat x=\frac{x}{|x|}. 
\end{equation}
The resolvent estimates \eqref{exist_local_1} or \eqref{exist_local_2}, combined with the Sommerfeld radiation condition \eqref{uni2}, ensure the existence and uniqueness of the solution $u$ of \eqref{local_Helmhotz} in specific $L^q$ spaces, where the source term $f$ is in the desired $L^p$ spaces, as given by \eqref{les_con0} or \eqref{eq:pq2}.

Moving into the fractional Helmholtz equation, it is shown in \cite{Zilberberg2026}*{Theorem 2.2} that, in dimension $d=1,2,3$, the outgoing fundamental solution of the equation
\begin{equation} \label{eq:uf}
	((-\Delta)^s - k^{2s}) u = f \ \ \mbox{in} \ \ \mathbb{R}^{d}
\end{equation}
satisfies the following Sommerfeld radiation condition (SRC):
\begin{equation} \label{eq:src}
	r^{\frac {d-1} 2} \Big( \frac {\partial u} {\partial r} - ik u \Big) \to 0 \ \ \mbox{as} \ \ r \to +\infty,
\end{equation}
or the generalized Sommerfeld radiation condition (GSRC):
\begin{equation} \label{fractional_Helmhotz}
	\int_{\R^d \backslash B_R} \langle x \rangle^{\delta-1} |\nabla u - \mathrm{i} k \hat x u|^2 \dif x < +\infty.
\end{equation}
Also, from \cite{Zilberberg2026}*{Theorems 3.12 \& 3.14} we know the equation \eqref{eq:uf}, when either the SRC \eqref{eq:src} or the GSRC \eqref{fractional_Helmhotz} is imposed, has a unique solution, called the limiting outgoing solution, given by (see \eqref{eq:Rks})
\begin{equation*}
	u := \lim_{\varepsilon\to 0^+} \big((-\Delta)^s-(k^{2s}+\mathrm{i}\varepsilon)\big)^{-1} f,
\end{equation*}
for $f \in L^{2,\delta}(\R^d) := \{ f \in L^1_{loc}(\R^d) \mid \int_{\R^d} \langle x \rangle^{2\delta} |f|^2 \dif x  < +\infty \}$.
We summarize the main results of \cite{Zilberberg2026} below.

\begin{prop}
	Let $d=1,2,3$, $0<s<1$, $k>0$, and let $G_{n,s}^k$ be the outgoing fundamental solution of $(-\Delta)^s - k^{2s}$.  
	\begin{enumerate}
		\item[(i)] The function $G_{d,s}^k$ itself satisfies \eqref{eq:src} and \eqref{fractional_Helmhotz}.
		 
		\item[(ii)] For any $f\in L^{2,\delta}(\mathbb{R}^d)$ with $\frac12<\delta<1$ and $s\ge \frac12$, the convolution $u = G_{d,s}^k * f$ is the unique solution in $H^{2s,-\delta}(\mathbb{R}^d)$ satisfying \eqref{fractional_Helmhotz} that solves $((-\Delta)^s - k^{2s})u = f$ in $\R^d$.
		
		\item[(iii)] For any compactly supported $f\in L^2(\mathbb{R}^d)$ and any $0<s<1$, the convolution $u = G_{d,s}^k * f$ is the unique solution in $H^{2s,-\delta}(\mathbb{R}^d)$ satisfying \eqref{eq:src} that solves $((-\Delta)^s - k^{2s})u = f$ in $\R^d$.
	\end{enumerate}
\end{prop}

\subsection{Fractional Schr\"{o}dinger operator}

We now discuss the solvability of the fractional Schr\"{o}dinger equation in the outgoing resolvent sense:
\begin{equation}\label{fractional_Schrodinger}
(-\Delta)^s u - k^{2s} u - V(x) u = f(x), \quad\mbox{in }\mathbb{R}^{d}.
\end{equation}
The authors in \cite{HYZ} (see Proposition 3.3 there) have dealt with this case as well, and as an application of the uniform estimate for the fractional resolvent \eqref{resol_frac_2}, they found that
\begin{prop}\label{resolvent_frac-schro}
Let $d\geq 3$, $\frac{d}{d+1}\leq s<\frac{d}{2}$,  $1<p<q<\infty$  are Lebesgue exponents satisfying \eqref{les_con},  $V\in L^{\frac{d}{2s}}(\mathbb{R}^d)$, and let $H=(-\Delta)^{s}+V(x)$. There exists a constant $c_0>0$ such that
if $\|V\|_{L^{\frac{d}{2s}}}\leq c_0$, then
\begin{equation}
	\|(H-z)^{-1}\|_{L^p-L^{q}}\leq C |z|^{\frac{d}{2s}(\frac1p-\frac{1}{q})-1},\,\,\, z\in \mathbb{C}\setminus\{0\}
\end{equation}
for $\max(\frac{2s}{d}, \frac{d+3}{2(d+1)})<\frac{1}{p}\leq\frac{d+2s}{2d}$.
\end{prop}
We restate the above proposition with the obvious tweak that it allows all $V \in L^{\infty}(\mathbb{R}^d)$ with compact support, for high frequencies. 
\begin{prop}\label{resolvent_frac-schro2}
Let $d\geq 3$, $\frac{d}{d+1}\leq s<\frac{d}{2}$,  $1<p<q<\infty$  are Lebesgue exponents satisfying \eqref{les_con},  $V\in L^{\infty}(\mathbb{R}^d)$ with compact support, and let $H=(-\Delta)^{s} - V(x)$. Then there exists a $\lambda_0>0$, such that for $\lambda>\lambda_0$, 
\begin{align} \label{resol_frac_schro}
	\sup_{0<\varepsilon<1}\|(H-(\lambda+\mathrm{i}\varepsilon))^{-1}\|_{L^{p}-L^{q}}\leq C|\lambda|^{\frac{d}{2s}(\frac{1}{p}-\frac{1}{q})-1},\,\,\lambda>\lambda_0,
\end{align}
where the constant $C(\lambda_0)>0$ is uniform w.r.t. $\lambda>\lambda_0$ and $\varepsilon$. 
\end{prop}

\begin{proof}
Let us denote 
\begin{equation}
	\mathcal{R}_{V,\lambda}^{s} = ((-\Delta)^{s} - V(x) - \lambda)^{-1}.
\end{equation}
Then it can be realized as 
\begin{equation} \label{resol_schr}
	\mathcal{R}_{V,\lambda}^{s} =
	\begin{cases} \mathcal R_\lambda^s (I - V \mathcal R_\lambda^s)^{-1} \\
		(I - \mathcal R_\lambda^s V)^{-1} \mathcal R_\lambda^s,
	\end{cases}
\end{equation}
where, $\mathcal{R}_\lambda^{s}=((-\Delta)^{s}-\lambda)^{-1}$; whenever $(I - V \mathcal{R}_\lambda^s)^{-1}$ exists.
Let us show that, $(I - V\mathcal{R}_\lambda^{s})^{-1}$ exists. By applying \eqref{resol_frac_2} and H\"{o}lder's inequality one obtains
\begin{align*}
	\|V((-\Delta)^{s}-(\lambda+\mathrm{i}\varepsilon))^{-1}\|_{L^{p}-L^{p}}\leq C|\lambda|^{\frac{d}{2s}(\frac{1}{p}-\frac{1}{q})-1}\|V\|_{L^{\infty}},
\end{align*}
then by choosing $\lambda>\lambda_0$ large enough for some $\lambda_0>0$ satisfying
\[
C|\lambda_0|^{\frac{d}{2s}(\frac{1}{p}-\frac{1}{q})-1}\|V\|_{L^{\infty}} <1,
\]
we can have
\begin{align} \label{resol_frac_schr2}
	\|(I - V((-\Delta)^s - (\lambda+\mathrm{i}\varepsilon))^{-1})^{-1}\|_{L^{q}-L^{q}}
	\leq 2.
\end{align}
Hence the claim follows by combining \eqref{resol_frac_schr2}, \eqref{resol_schr}, and \eqref{resol_frac_2}.
\end{proof}

As a remark, we mention that in the case of classical Schr\"{o}dinger operator $H=-\Delta+V$ due to Agmon \cite{AgSh}, known as {\it the limiting absorption principle}, we have the following resolvent estimate 
\begin{align}
\sup_{0<\varepsilon<1}\|(-\Delta+V-(\lambda+i\varepsilon))^{-1}\|_{L^{2, \delta}-L^{2, -\delta}}\leq C(\lambda_0)\lambda^{-\frac{1}{2}}, ~~~\lambda>\lambda_0>0,
\end{align}
where  $|V(x)|\leq C(1+|x|)^{-1-}$ and $\delta>\frac12$ and  
$L^{2, \delta}$ for $-\infty<\delta<\infty$, is the usual weighted $L^2$ space defined as the completion of $C^\infty_c(\mathbb{R}^d)$  w.r.t.~the norm $\|f\|_{L^{2, \delta}} := \|(1+|x|^2)^\frac{\delta}{2}\,f\|_{L^2}$. 

Thanks to Goldberg and Schlag \cite{GoSc}, we also have a $L^p$ version of the limiting absorption principle for the three dimensional Schr\"{o}dinger operators. More specifically, for any given $\lambda_0>0$, they proved
\begin{align}
\sup_{0<\varepsilon<1}\|(-\Delta+V-(\lambda+i\varepsilon))^{-1}\|_{L^{\frac43}(\mathbb{R}^3)-L^4(\mathbb{R}^3)}\leq C(\lambda_0)\lambda^{-\frac14}, ~~~\lambda>\lambda_0,
\end{align}
where  $V\in L^p(\mathbb{R}^3) \cap L^{\frac32}(\mathbb{R}^3)$, $p>\frac32$.

\section{Direct problem}

\subsection{Direct problem}

Thanks to the discussion in the previous section, now we have the solvability of the fractional Schr\"{o}dinger equation
$((-\Delta)^s u - k^{2s} - V(x)) u = 0$ in $\mathbb{R}^d$ in certain $L^q$ spaces and for certain $s\in (0,1)$.  Specifically, we state 
\begin{thm}\label{dir_pb}
	Let $d\geq 3$, $\frac{d}{d+1}\leq s<\frac{d}{2}$, and $V\in L^{\infty}(\mathbb{R}^d)$ with compact support $\Omega \subset \mathbb{R}^d$.
	Let us fix some $1<p<q<\infty$ satisfying \eqref{les_con}. Then there exists a $k_0>0$, such that for $k>k_0$, the scattered field $u^{\mathrm{sc}}_g$ corresponding to the incident field
\begin{equation*} 
	u^{\mathrm{in}}_g = \int_{\mathbb{S}^{d-1}} e^{\mathsf{i}kx \cdot \theta} g(\theta) \dif s(\theta)
\end{equation*}
exists in $L^q(\mathbb{R}^d)$ as the outgoing limiting resolvent solution of
\begin{equation} \label{frac_scat_direct}
	\begin{aligned}
		\LC(-\Delta)^s - k^{2s} - V(x) \RC u^{\mathrm{sc}}_g & = V(x)u^{\mathrm{in}}_g  \in L^p (\mathbb{R}^d).
	\end{aligned}
\end{equation}
Moreover, there is a uniform constant $C_{p,q}>0$ such that 
\begin{equation}\label{estimate_usc}
	||u^{\mathrm{sc}}_g||_{L^{q}(\mathbb{R}^{d})}
	\leq C_{p,q}\,k^{d(\frac{1}{p}-\frac{1}{q})-2s}||Vu^{\mathrm{in}}_g||_{L^{\infty}(\Omega)}.
\end{equation}
Hence, the total field $u_g = u^{\mathrm{in}}_g+u^{\mathrm{sc}}_g \in  L^\infty+ L^q$ is the outgoing limiting resolvent solution of $-(-\Delta)^s u+ \LC k^{2s}+ V(x)\RC u=0$ in $\mathbb{R}^d$.
\end{thm} 
This result concludes the discussion of the direct problem of fractional Schr\"{o}dinger equation \eqref{eq:1}, next we introduce \emph{scattering amplitude} and the related inverse problem. 

\subsection{Scattered field and scattering amplitude} 

Let us recall (see \eqref{1.5}) that the \emph{scattered field} is given by
\begin{equation}\label{sct_fld}
u^{\mathrm{sc}}_g(x)=\int_{\mathbb{R}^d}\Phi_s(x,y)V(y)u_g(y)\,\mathrm{d} y, \ \ \text{ for } x \in \mathbb{R}^d.
\end{equation}
Let us further recall from \eqref{resols1}  that 
\begin{align}
	\Phi_{\pm,\, s}(x-y)
	&  =\frac{k^{2(1-s)}}{s}\Phi_{\pm,\, 1}(x-y) \\
	& \quad + \frac{\sin{s\pi}}{\pi}\int_{0}^{\infty}\frac{\lambda^{s}(\lambda-\Delta)^{-1}\delta_y(x)} {\lambda^{2s}-2\lambda^{s}k^{2s}\cos{s\pi}+k^{4s}}\,  \mathrm{d}\lambda,
\end{align}
and we are interested in providing a decay for the second term when $|x|\to \infty.$
In particular, we want to show
\begin{equation}\label{sasy}
\Phi_{\pm,\, s}(x)\sim \frac{k^{2(1-s)}}{s}\Phi_{\pm,\,1}(x),  
\end{equation}
where $f\sim g$ is defined through $L^2$ integrability at infinity, i.e.,
\begin{equation} \label{decay} f\sim g \Longleftrightarrow \lim_{R\to\infty} \frac{1}{R}\int_{1<|x|< R}|f(x)-g(x)|^2\,\mathrm{d} x=0.
\end{equation}
By using the following estimates
\begin{equation}
\big| (\lambda-\Delta)^{-1}\delta_y(x) \big| \leq
\left\{\begin{array}{cl}
	C\lambda^{\frac{d-2}{4}}|x-y|^{-\frac{d-2}{2}}e^{-\sqrt{\lambda}|x-y|}, \, \sqrt{\lambda}|x-y|>1, \\
	C|x-y|^{2-d}, \, \sqrt{\lambda}|x-y|\leq 1,
\end{array}\right.
\end{equation}
a direct computation which can be found in \cite{HYZ}*{page 12-13} yields 
\begin{align*}
\int_{0}^{\infty} \Big|\, \frac{\lambda^{s}(\lambda-\Delta)^{-1}\delta_y(x)}
{\lambda^{2s}-2\lambda^{s}k^{2s}\cos{s\pi}+k^{4s}}\, \Big|\,\mathrm{d}\lambda\leq C|x-y|^{-d-2s}, \quad \text{if}\,\, |x-y|>1.
\end{align*}
Hence \eqref{sasy} follows. Further by taking \eqref{31asy} into account, we have
\begin{align}\label{sasy2}
\Phi_{\pm,\, s}(x)\sim \frac{k^{2(1-s)}}{s}\,k^{\frac{d-3}{2}}\frac{e^{-\frac{\mathrm{i}\pi(d-3)}{4}}}{2^{\frac{d+1}{2}}\,\pi^{\frac{d-1}{2}}}\,\,\frac{e^{\pm \mathsf{i}k|x|}}{|x|^{\frac{d-1}{2}}}.
\end{align}
\begin{prop}
\label{prop1}
Let $d\geq 2$ and $0 < s<  1$. Then the fundamental solution $\Phi_{\pm,\, s}(x)$ of the fractional Helmholtz equation
$\LC (-\Delta)^s - k^{2s}\RC \Phi_{\pm,\, s}(x) = \delta_0(x)$  in  $\mathbb{R}^d$,  satisfies \eqref{sasy2}, in the sense of \eqref{decay}.
\end{prop} 
As an application, the integral equation \eqref{sct_fld} can be simplified to \begin{equation}\label{lp}
u^{\mathrm{sc}}_g(x)\sim\frac{k^{2(1-s)}}{s}\,k^{\frac{d-3}{2}}\frac{e^{-\frac{\mathsf{i}\pi(d-3)}{4}}}{2^{\frac{d+1}{2}}\,\pi^{\frac{d-1}{2}}}\,\int_{\mathbb{R}^d}\frac{e^{\pm \mathrm{i}k|x-y|}}{|x-y|^{\frac{d-1}{2}}}\,V(y)u_g(y)\,\mathrm{d}y.
\end{equation}
Let us introduce the function $u^{\infty}_g$ to be called the \emph{scattering amplitude} or the \emph{far field pattern}, which is of the form
\begin{equation}
u^{\infty}_g(\hat x)=\int_{\mathbb{R}^d}e^{-\mathsf{i}k\hat x\cdot y}V(y) u_g(y)\,\mathrm{d}y, \quad \hat x \in \mathbb{S}^{d-1}. \label{1.8}
\end{equation}
Subsequently, a standard calculation on the l.h.s. of \eqref{lp}, which can be found in \cite{GE}*{page 92}, leads to the following relation linking the scattered field and the scattering amplitude: 
\begin{equation}
u^{\mathrm{sc}}_g(x)\sim \frac{k^{2(1-s)}}{s}\,k^{\frac{d-3}{2}}\frac{e^{-\frac{\mathsf{i}\pi(d-3)}{4}}}{2^{\frac{d+1}{2}}\,\pi^{\frac{d-1}{2}}}\,\,\frac{e^{ \pm\mathsf{i}k|x|}}{|x|^{\frac{d-1}{2}}}\, u^{\infty}_g(\hat x). \label{1.7}
\end{equation}
Then the total field $u_g$ solving \eqref{eq:1} satisfies
\begin{equation}\label{asmp50}
u_g(x) \sim u^{\mathrm{in}}_g(x)+ \frac{k^{2(1-s)}}{s}\,k^{\frac{d-3}{2}}\frac{e^{-\frac{\mathsf{i}\pi(d-3)}{4}}}{2^{\frac{d+1}{2}}\,\pi^{\frac{d-1}{2}}}\,\,\frac{e^{ \pm\mathsf{i}k|x|}}{|x|^{\frac{d-1}{2}}}\, u^{\infty}_g(\hat x).
\end{equation} 
It is noted here that the scattering amplitude $u^{\infty}_g(\hat x)$ corresponds to the Herglotz wave as an incident field, and further it can be readily seen as 
\begin{equation}
u^{\infty}_g(\hat x)=\int_{\mathbb{S}^{d-1}}u^{\infty}(k,\hat x,\theta)g(\theta)\, ds(\theta), \quad\text{ for } \hat x:=\frac{x}{|x|}\in \mathbb{S}^{d-1}, \ g \in L^{2}(\mathbb{S}^{d-1}), \label{sc_amp_H}
\end{equation}    
where  for $(k, \hat x, \theta) \in (0,\infty)\times\mathbb{S}^{d-1}\times \mathbb{S}^{d-1}$,
\begin{equation}\label{sc_amp_3}
u^{\infty}(k,\hat x, \theta)=\int_{\mathbb{R}^d}e^{-\mathsf{i}k\hat x\cdot y}V(y)\underbrace{\left( e^{iky\cdot\theta}+ u^{\mathrm{sc}}(y)\right)}_{u(y)}\,\mathrm{d}y,
\end{equation}
the scattering amplitude corresponds to the incident plane waves $e^{\mathrm{i}kx \cdot \theta}$.

\section{Inverse problem}

Now we discuss the inverse problem that we are interested in. 
Let $V\in L^\infty(\Omega)$ where $\Omega$ is a bounded domain in $\mathbb{R}^d$ with $C^\infty$-smooth boundary, $d\geq 3$,
such that $\mathbb{R}^d\setminus\Omega$ is connected. Extending $V(x)$ by zero outside of $\Omega$ we 
consider the fractional inverse scattering problem in $\mathbb{R}^d$ for the equation \eqref{eq:1}, i.e., the problem of recovering $V(x)$ from the scattering amplitude $u^{\infty}(k,\hat x, \theta)$ (cf.~\eqref{sc_amp_3}). 

\medskip
Let $u^{\infty}_j(k,\hat x, \theta)$ be the scattering amplitude for the total field $u_{j}$ solving 
\begin{equation} \label{1.9}
	-(-\Delta)^s u_{j}+\LC k^{2s} +V_j(x) \RC u_{j}=0  \text{ in }  \mathbb{R}^d, \quad j=1,2,
\end{equation}
where
\begin{equation*}
	u_{j}(x)=e^{\mathrm{i}kx\cdot\theta}+u^{\mathrm{sc}}_{j}(x) \text{ in }\R^d,
\end{equation*}
and 
\begin{equation} \label{1.72}
	u^{\mathrm{sc}}_j(x)\sim \frac{k^{2(1-s)}}{s}\,k^{\frac{d-3}{2}}\frac{e^{-\frac{\mathsf{i}\pi(d-3)}{4}}}{2^{\frac{d+1}{2}}\,\pi^{\frac{d-1}{2}}}\,\,\frac{e^{ \mathsf{i}k|x|}}{|x|^{\frac{d-1}{2}}}\, u^{\infty}_j(k,\hat x, \theta).
\end{equation}
Suppose that
\(
u^{\infty}_1(k,\hat x, \theta)=u^{\infty}_2(k,\hat x, \theta),
\)
then we would like to study whether it implies $V_1 = V_2$.

\subsection{Rellich Lemma \& Local Problem}

For readers' convenience we restate the classical Rellich lemma which can be found in e.g.~\cite{GE}*{Lemma 35.2}.

\begin{lem} \label{Rellich}
Let $k\neq 0$, and let $f\in L^2(\mathbb{R}^d)$ have compact support.
$u\in L^2(\mathbb{R}^d)$ solves
\begin{equation}
	\Delta u + k^{2}u =f, \quad\mbox{in }\mathbb{R}^d,
\end{equation}
Then $u$ has compact support, and $\supp u\subseteq \supp f$.
\end{lem}

\begin{lem} \label{lem:w12}
Let $w_j~(j=1,2)$ be the solution of 
\begin{equation}\label{uv}
	\Delta w_{j}+  k^{2} w_{j} =-V_ju_{j}  \quad\text{ in } \mathbb{R}^d,
\end{equation}
where $u_{j}$'s are the total fields solving \eqref{1.9}, and $V_j$'s are the bounded potentials with supports in $\overline{\Omega}$.
Then $u^{\infty}_1(k,\hat x,\theta) = u^{\infty}_2(k,\hat x,\theta)$ for all $\hat x \in \mathbb S^{d-1}$ implies $w_1 - w_2 \in L^2(\mathbb R^d)$ with compact support.
\end{lem}

\begin{proof}
Using the fundamental solution of the Helmholtz operator, we write down the (outgoing) solution as
\begin{equation}\label{vjg}
	w_{j}(x)=-\int_{\mathbb{R}^d}\Phi_{\pm,1}(x,y)V_j(y)u_{j}(y)\,\mathrm{d} y, \quad\text{ for } x \in \ \mathbb{R}^d
\end{equation}
where $\Phi_{\pm,1}$ is given in \eqref{31} and \eqref{313}, with its asymptotics in \eqref{31asy}. Similar to deriving \eqref{1.72}, we can further infer that
\begin{equation}
	w_{j}(x)\sim -k^{\frac{d-3}{2}}\frac{e^{-\frac{\mathsf{i}\pi(d-3)}{4}}}{2^{\frac{d+1}{2}}\,\pi^{\frac{d-1}{2}}}\,\,\frac{e^{\pm \mathsf{i}k|x|}}{|x|^{\frac{d-1}{2}}}\, u^{\infty}_{j}(k,\hat x,\theta),\quad j=1,2,
\end{equation}
and using the hypothesis $u^{\infty}_{1}=u^{\infty}_{2}$, we have
$w_{1}-w_{2} \in L^2(\mathbb{R}^d)$.

Denote $f := (V_2u_{2}-V_1u_{1}) \in L^2(\mathbb{R}^d)$, which has compact support in $\overline{\Omega}$.
Moreover, $w_{1}-w_{2}$ solves
\begin{equation}\label{h2}
	\Delta (w_{1}-w_{2}) + k^{2} (w_{1}-w_{2}) =f\quad\mbox{in }\mathbb{R}^d,
\end{equation}
so applying Lemma \ref{Rellich} to $(w_{1}-w_{2})$ gives
\begin{equation}\label{w12}
	w_{1}-w_{2} = 0 \quad\mbox{in }\mathbb{R}^d\setminus\overline{\Omega}.
\end{equation}
Let $\Omega\Subset B(0,\sigma)=B$ for some $\sigma>0$. Note that, $(w_{1}-w_{2})\in H^2(B)$ by the elliptic regularity of the solution of \eqref{h2}. Furthermore, thanks to \eqref{w12}, we can conclude $w_1 - w_2 \in H^2_0(B)$.
The proof is done.
\end{proof}

We are ready to prove Theorem \ref{thm:2}.

\begin{proof}[Proof of Theorem \ref{thm:2}]
Let $m\in\mathbb{R}^d$ be arbitrary.
Choose $l \in\mathbb{R}^d$ such that $m \cdot l = 0$ and $|m + l| > k_0$.
Denote
\begin{equation} \label{rho}
	\rho := m+l, \quad
	k := |\rho| = \sqrt{|m|^2 + |l|^2}, \quad
	\theta := \frac {m-l} {|m-l|}.
\end{equation}
It can be checked that $k \in (k_0,+\infty)$ and $\theta \in \mathbb S^{d-1}$.

Let $w_j~(j=1,2)$ solve
\(
(-\Delta - k^2) w_j = V_j u_j
\)
in $\mathbb{R}^d$.
By Lemma \ref{lem:w12} we see $w_1 - w_2 \in L^2(\mathbb R^d)$ with compact support.
Then integration by parts gives
\begin{align}
	\int_{B} e^{\mathrm{i}\rho\cdot x} (V_2 u_2 - V_1 u_1)
	& = \int_{B} e^{\mathrm{i}\rho\cdot x}(\Delta+k^2) (w_1 - w_2) \nonumber \\
	& = \int_{B} (\Delta+k^2) e^{\mathrm{i}\rho\cdot x} (w_1 - w_2)
	= 0. \label{fr_eq}
\end{align}
Substituting $u_j = e^{\mathrm{i}k \theta \cdot x} + u^{\mathrm{sc}}_j$ into \eqref{fr_eq}, we can obtain
\begin{equation} \label{eq:Vu}
	\int_{B} e^{\mathrm{i}m \cdot x} (V_1 - V_2) = -\int_{B} e^{\mathrm{i}\rho\cdot x} (V_1 u^{\mathrm{sc}}_1 - V_2 u^{\mathrm{sc}}_2).
\end{equation}

Next, we are to use Proposition \ref{prop5} to estimate $\|u^{\mathrm{sc}}_j\|~(j=1,2)$.
To use Proposition \ref{prop5}, we need to choose the values of $p,q,s$ particularly so that requirements in Proposition \ref{prop5} are all satisfied.
We set
\begin{equation*}
	\left\{\begin{aligned}
		\frac 1 q & := \frac 1 2 \Big( \frac {d+1-3.8s} {2d} + \frac {d-1} {2d} \Big) = \frac {d - 1.9s} {2d}, \\
		\frac 1 p &:= \frac 1 2 \Big[ \max\{ \frac 2 {d+1} + \frac 1 q, \frac {d+1} {2d} \} + (\frac {2s} d + \frac 1 q) \Big].
	\end{aligned}\right.
\end{equation*}
We fix $d \geq 3$ and choose $s$ from
\[
s \in (\frac{d}{d+1}, \frac{d}{2}).
\]
Here, note that we require $s$ to be strictly larger than $\frac d {d+1}$, which is slightly different from the statement in Proposition \ref{prop5}.
But this is necessary for the convergence we shall see soon.
Under these choices, it can be checked that
\(
\frac {d+1-3.8s} {2d} < \frac {d-1} {2d},
\)
so $\frac 1 q < \frac {d-1} {2d}$.
Moreover, we can show
\(
\frac 2 {d+1} + \frac 1 q < \frac {2s} d + \frac 1 q
\)
and
\(
\frac {d+1} {2d} < \frac {2s} d + \frac 1 q,
\)
so $\frac1 p$ belongs to the open interval $(\frac 2 {d+1} + \frac 1 q, \frac {2s} d + \frac 1 q)$.
Now all the conditions in Proposition \ref{prop5} are satisfied, with an additional benefit that
\begin{equation} \label{eq:ds}
	\frac d 2 (\frac 1 p - \frac 1 q) - s
	= s[\frac d {2s} (\frac 1 p - \frac 1 q) - 1] < 0.
\end{equation}
Recall that $u^{\mathrm{sc}}_j$ satisfies
\[
((-\Delta)^s - k^{2s}) u^{\mathrm{sc}}_j = V_j e^{\mathrm{i}k \theta \cdot x} + V_j u^{\mathrm{sc}}_j.
\]
Denote $\mathcal R_k := ((-\Delta)^s - (k^{2s}+\mathrm{i}0))^{-1}$, then $u^{\mathrm{sc}}_j$ can be represented as Neumann series
\begin{equation} \label{eq:neum}
	u^{\mathrm{sc}}_j
	= (I - \mathcal R_k V_j)^{-1} \mathcal R_k V_j e^{\mathrm{i}k \theta \cdot x}
	= \mathcal R_k \sum_{\ell \geq 0} (V_j \mathcal R_k)^\ell (V_j e^{\mathrm{i}k \theta \cdot x}),
\end{equation}
provided that $\norm{\mathcal R_k V_j}_{L^q-L^q} < 1$.
To validate this, for any $f \in L^q(\R^d)$ we can compute
\begin{align*}
	\norm{\mathcal R_k V_j f}_{L^q-L^q}
	& \leq \norm{\mathcal R_k}_{L^p-L^q} \norm{V_j f}_{L^p}
	\leq \norm{\mathcal R_k}_{L^p-L^q} \norm{V_j}_{L^{pq/(q-p)}} \norm{f}_{L^q} \\
	& \leq C_{p,q} k^{d(\frac 1 p - \frac 1 q) - 2s} \norm{V_j}_{L^\infty} \norm{f}_{L^q},
\end{align*}
where we used \eqref{resol_frac_2}.
Because of \eqref{eq:ds}, we see when $k$ is larger than certain $k_0$, which depends on $p,q,d,s,\norm{V_j}_{L^\infty}$, we can have $\norm{\mathcal R_k V_j}_{L^q-L^q} < 1$.
Therefore, by Proposition \ref{prop5} we can conclude from \eqref{eq:neum} that
\begin{align*}
	\|u^{\mathrm{sc}}_j\|_{L^{q}(\mathbb{R}^{d})}
	& \leq \norm{\mathcal R_k}_{L^p-L^q} \norm{(I - V_j \mathcal R_k)^{-1}}_{L^q-L^q} \norm{V_j e^{\mathrm{i}k \theta \cdot x}}_{L^q(\Omega)} \\
	& \leq \frac {\norm{\mathcal R_k}_{L^p-L^q}} {1 - \norm{V_j \mathcal R_k}_{L^q-L^q}} \norm{V_j}_{L^q(\Omega)} \\
	& \leq C (m^2+l^2)^{\frac d 2 (\frac1 p - \frac 1 q) - s} \|V_j\|_{L^{\infty}(\Omega)}.
\end{align*}
Combining this with \eqref{eq:Vu} we can have
\begin{align*}
	\big| \int_{B} e^{\mathrm{i}m \cdot x} (V_1 - V_2) \big|
	& = \big| \int_{B} e^{\mathrm{i}\rho\cdot x} (V_1 u^{\mathrm{sc}}_1 - V_2 u^{\mathrm{sc}}_2) \big|
	\leq \sum_{j=1}^2 \norm{V_j u^{\mathrm{sc}}_j}_{L^1(\Omega)} \\
	& \leq \sum_{j=1}^2 \norm{V_j}_{L^{q/(q-1)}(\Omega)} \norm{u^{\mathrm{sc}}_j}_{L^q(\R^d)} \\
	& \leq C (m^2+l^2)^{\frac d 2 (\frac1 p - \frac 1 q) - s} (\|V_1\|_{L^{\infty}(\Omega)}^2 + \|V_2\|_{L^{\infty}(\Omega)}^2).
\end{align*}
Because of \eqref{eq:ds}, by taking the length of $l$ to infinity we arrive at
\begin{equation*}
	\int_{B} e^{2\mathrm{i}m\cdot x} (V_1 - V_2) = 0.
\end{equation*}
As $m\in\mathbb{R}^d$ was arbitrary, we conclude that $V_1 = V_2$.
The proof is done.
\end{proof}

\section*{Acknowledgments}

The research of S.~M.~is partially supported by the NSFC under the grant No.~12301540. T.~G.~is grateful to Gunther Uhlmann and Catalin Carstea for discussions at IAS, HKUST.
The authors are grateful to the anonymous reviewers for their helpful feedback.



\vspace{.3cm}
\noindent \textbf{Data Availability:} The authors shall permit all the data underlying the findings of this manuscript to be shared
by any researchers or groups who are interested in the article.

\vspace{.3cm}
\noindent \textbf{Conflict of Interest:} The authors declare that there is no conflict of interest regarding the publication of this paper.

\begin{appendix}

\section{Rellich lemma and solution construction} \label{Appendix}

Rellich's lemma has opposite results for the local and non-local cases.

\begin{lem} \label{Rellich_theorem}
	Let $k > 0$ and $u\in L^2(\mathbb{R}^d)$ be a solution of 
	\begin{equation} \label{eqn1}
		(-\Delta)^s u - k^{2s} u = f \quad \mbox{in} \quad \mathbb{R}^d,
	\end{equation}
	where $f\in L^2(\mathbb{R}^d)$ and has compact support. Then
	\begin{enumerate}
		\item when $s = 1$, $\supp u$ is compact, and $\supp u\subseteq \supp f$.
		
		\item when $0 < s < 1$, $u$ cannot be compactly supported unless $f \equiv 0$.
	\end{enumerate} 
\end{lem}

\begin{proof}
	The case $s = 1$ is the classical Rellich's lemma, whose proof can be found in e.g.~\cite{GE}*{Lemma 35.2}.
	
	For the case $0 < s < 1$, we assume $f$ is non-zero.
	Then if $u$ is compactly supported, then we can find a non-empty domain $W$ such that $u = f = 0$ in $W$, and due to $(-\Delta)^s u = k^{2s} u + f$, we can have $(-\Delta)^s u = 0$ in $W$, too.
	Therefore, by \cite{GSU20}*{Theorem 1.2} we conclude $u \equiv 0$ in $\mathbb R^d$, and thus $f \equiv 0$.
	We meet a contradiction.
	This indicates $u$ cannot be compactly supported.
	We can even conclude that $u$ cannot be equal to zero in any bounded domain contained in $\supp f$.
	The proof is done.
\end{proof}

In the following part, for a given non-zero, compactly supported, $L^2$-integrable function $f$, by supposing \textit{a priori} there is a solution $u \in L^2(\R^d)$ of \eqref{eqn1},
we can construct such $u$.

\smallskip

Let us take the case $s=\frac{1}{2}$ first. Taking the Fourier transform in \eqref{eqn1} gives
\begin{equation*} 
	(|\xi|-k)\widehat{u}(\xi) = \widehat{f}(\xi).
\end{equation*}
Define $w \in H^1(\mathbb{R}^d)$ solving the following elliptic equation with the source term $u\in L^2(\R^d)$:
\begin{equation} \label{eqn6}
	(-\Delta)^{\frac{1}{2}}w + kw =u \quad \mbox{in} \quad \mathbb{R}^d, \quad k>0.
\end{equation}
Using the Fourier transform, $w$ can be solved as
\(
w = \mathcal F^{-1} \big\{(|\xi| + k)^{-1} \widehat{u}(\xi) \big\}.
\)
This immediately shows $\widehat{w}\in L^2(\R^d)$ and $|\xi| \widehat{w}(\xi) \in L^2(\R^d)$, i.e.~$w \in H^1(\mathbb{R}^d)$, thanks to our hypothesis $u\in L^2$.
Next, basic algebra shows  
\begin{equation}
	\label{eqn4}
	(|\xi|^2-k^2)\widehat{w}(\xi)
	= (|\xi|-k)\widehat{u}(\xi)
	= \widehat{f}(\xi),
\end{equation}
which gives 
\begin{equation} \label{w-eqn}
	(-\Delta - k^2) w = f \quad \mbox{in} \quad \mathbb{R}^d.
\end{equation}
The Helmholtz equation \eqref{w-eqn} can be solved using standard resolvent arguments.
Then we can obtain $u$ using \eqref{eqn6}.

%
%
%

\smallskip

For $s\neq \frac{1}{2},$ we define \begin{equation} \label{eqn7}
	\widehat{w}(\xi) = \frac{|\xi|^{2s}-k^{2s}}{|\xi|^2-k^2}\,\widehat{u}(\xi).
\end{equation}
L'H\^opital's rule shows
\[
\lim_{|\xi| \to k} \frac {|\xi|^{2s}-k^{2s}} {|\xi|^2-k^2}
= sk^{2s-2} > 0,
\]
so $\frac{|\xi|^{2s}-k^{2s}}{|\xi|^2-k^2} \neq 0$ for all $\xi\in\mathbb{R}^d$.
In fact, $\frac{|\xi|^{2s}-k^{2s}}{|\xi|^2-k^2} > 0$ due to continuity, and it is a bounded function with respect to $\xi$, i.e.~$\frac {|\xi|^{2s} - k^{2s}} {|\xi|^2 - k^2} \in L^\infty(\R^d)$.
Now using our hypothesis $u \in L^2(\R^d)$, the relation \eqref{eqn7} implies $w \in L^2(\mathbb{R}^d)$, and it solves
\begin{equation} \label{w-eqnn}
	(-\Delta - k^2) w = f \quad \mbox{in} \quad \mathbb{R}^d.
\end{equation}
Solving the Helmholtz equation \eqref{w-eqnn} gives us $w \in H^2(\R^d)$.
Fix a cutoff function $\chi \in C^\infty_c(\R^d)$ with $\chi(\xi) = 1$ near $|\xi| = k$, then \eqref{eqn7} implies
\begin{equation}
	\widehat{u}(\xi)
	= \underbrace{\chi(\xi) \frac {|\xi|^2-k^2} {|\xi|^{2s} -k^{2s}}}_{\in L^\infty} \, \underbrace{\widehat{w}(\xi)}_{\in L^2} + \underbrace{(1-\chi(\xi)) \frac {1 - k^2/|\xi|^2} {|\xi|^{2s}-k^{2s}}}_{\in L^\infty} \, \underbrace{|\xi|^2 \widehat{w}(\xi)}_{\in L^2}.
\end{equation}
We see $u\in L^2(\R^d)$.
The construction is done.

\end{appendix}


\begin{bibdiv}
	\begin{biblist}
		
		\bib{AgSh}{article}{
			author={Agmon, Shmuel},
			title={Spectral properties of {S}chr\"{o}dinger operators and scattering
				theory},
			date={1975},
			ISSN={0391-173X},
			journal={Ann. Scuola Norm. Sup. Pisa Cl. Sci. (4)},
			volume={2},
			number={2},
			pages={151\ndash 218},
			url={http://www.numdam.org/item?id=ASNSP_1975_4_2_2_151_0},
			review={\MR{397194}},
		}
		
		\bib{barcelo2020uniqueness}{article}{
			author={Barcel{\'o}, Juan~Antonio},
			author={Castro, Carlos},
			author={Luque, Teresa},
			author={Mero{\~n}o, Cristobal~J.},
			author={Ruiz, Alberto},
			author={de~la Cruz~Vilela, Mar{\'{\i}}a},
			title={Uniqueness for the inverse fixed angle scattering problem},
			date={2020},
			ISSN={0928-0219},
			journal={J. Inverse Ill-Posed Probl.},
			volume={28},
			number={4},
			pages={465\ndash 470},
		}
		
		\bib{bhattacharyya2021inverse}{article}{
			author={Bhattacharyya, S.},
			author={Ghosh, T.},
			author={Uhlmann, G.},
			title={Inverse problems for the fractional-{L}aplacian with lower order
				non-local perturbations},
			date={2021},
			ISSN={0002-9947,1088-6850},
			journal={Trans. Amer. Math. Soc.},
			volume={374},
			number={5},
			pages={3053\ndash 3075},
			url={https://doi.org/10.1090/tran/8151},
			review={\MR{4237942}},
		}
		
		\bib{borg1946umkehrung}{article}{
			author={Borg, G\"oran},
			title={Eine {U}mkehrung der {S}turm-{L}iouvilleschen {E}igenwertaufgabe.
				{B}estimmung der {D}ifferentialgleichung durch die {E}igenwerte},
			date={1946},
			ISSN={0001-5962,1871-2509},
			journal={Acta Math.},
			volume={78},
			pages={1\ndash 96},
			url={https://doi.org/10.1007/BF02421600},
			review={\MR{15185}},
		}
		
		\bib{BSW}{book}{
			author={B\"{o}ttcher, Bj\"{o}rn},
			author={Schilling, Ren\'{e}},
			author={Wang, Jian},
			title={L\'{e}vy matters. {III}},
			series={Lecture Notes in Mathematics},
			publisher={Springer, Cham},
			date={2013},
			volume={2099},
			ISBN={978-3-319-02683-1; 978-3-319-02684-8},
			url={https://doi.org/10.1007/978-3-319-02684-8},
			note={L\'{e}vy-type processes: construction, approximation and sample
				path properties, With a short biography of Paul L\'{e}vy by Jean Jacod,
				L\'{e}vy Matters},
			review={\MR{3156646}},
		}
		
		\bib{biccari2025boundary}{article}{
			author={Biccari, Umberto},
			author={Warma, Mahamadi},
			author={Zuazua, Enrique},
			title={Boundary observation and control for fractional heat and wave
				equations},
			date={2025},
			journal={arXiv preprint arXiv:2504.17413},
		}
		
		\bib{M}{book}{
			author={Carracedo, Selso~M.},
			author={Alix, Miguel~S.},
			title={The theory of fractional powers of operators},
			series={North-Holland Mathematics Studies},
			publisher={North-Holland Publishing Co.},
			address={Amsterdam},
			date={1995},
			volume={187},
			ISBN={0-444-82074-3},
			review={\MR{1321446}},
		}
		
		\bib{cakoni2006qualitative}{book}{
			author={Cakoni, Fioralba},
			author={Colton, David},
			title={Qualitative methods in inverse scattering theory. {An}
				introduction},
			series={Interact. Mech. Math.},
			publisher={Berlin: Springer},
			date={2006},
			ISBN={3-540-28844-9},
		}
		
		\bib{colton1998inverse}{book}{
			author={Colton, David},
			author={Kress, Rainer},
			title={Inverse acoustic and electromagnetic scattering theory},
			edition={4th expanded edition},
			series={Appl. Math. Sci.},
			publisher={Cham: Springer},
			date={2019},
			volume={93},
			ISBN={978-3-030-30350-1; 978-3-030-30351-8},
		}
		
		\bib{CS07}{article}{
			author={Caffarelli, Luis},
			author={Silvestre, Luis},
			title={An extension problem related to the fractional {L}aplacian},
			date={2007},
			ISSN={0360-5302},
			journal={Comm. Partial Differential Equations},
			volume={32},
			number={7-9},
			pages={1245\ndash 1260},
			url={https://doi.org/10.1080/03605300600987306},
			review={\MR{2354493}},
		}
		
		\bib{das2025fractional}{article}{
			author={Das, Saumyajit},
			author={Ghosh, Tuhin},
			title={Fractional Borg-Levinson problem with small non-negative
				potential of small growth},
			date={2025},
			journal={arXiv preprint arXiv:2508.06998},
		}
		
		\bib{deift1979inverse}{article}{
			author={Deift, P.},
			author={Trubowitz, E.},
			title={Inverse scattering on the line},
			date={1979},
			ISSN={0010-3640,1097-0312},
			journal={Comm. Pure Appl. Math.},
			volume={32},
			number={2},
			pages={121\ndash 251},
			url={https://doi.org/10.1002/cpa.3160320202},
			review={\MR{512420}},
		}
		
		\bib{eskin1992inverse}{incollection}{
			author={Eskin, G.},
			author={Ralston, J.},
			title={Inverse backscattering},
			date={1992},
			volume={58},
			pages={177\ndash 190},
			url={https://doi.org/10.1007/BF02790363},
			note={Festschrift on the occasion of the 70th birthday of Shmuel
				Agmon},
			review={\MR{1226942}},
		}
		
		\bib{ER95}{article}{
			author={Eskin, G.},
			author={Ralston, J.},
			title={Inverse scattering problem for the {S}chr\"odinger equation with
				magnetic potential at a fixed energy},
			date={1995},
			ISSN={0010-3616,1432-0916},
			journal={Comm. Math. Phys.},
			volume={173},
			number={1},
			pages={199\ndash 224},
			url={http://projecteuclid.org/euclid.cmp/1104274526},
			review={\MR{1355624}},
		}
		
		\bib{GE}{book}{
			author={Eskin, Gregory},
			title={Lectures on linear partial differential equations},
			series={Graduate Studies in Mathematics},
			publisher={American Mathematical Society, Providence, RI},
			date={2011},
			volume={123},
			ISBN={978-0-8218-5284-2},
			url={https://doi.org/10.1090/gsm/123},
			review={\MR{2809923}},
		}
		
		\bib{gel1951determination}{article}{
			author={Gel\cprime~fand, I.~M.},
			author={Levitan, B.~M.},
			title={On the determination of a differential equation by its spectral
				function},
			date={1951},
			journal={Doklady Akad. Nauk SSSR (N.S.)},
			volume={77},
			pages={557\ndash 560},
			review={\MR{43315}},
		}
		
		\bib{ghosh2021non}{article}{
			author={Ghosh, Tuhin},
			title={A non-local inverse problem with boundary response},
			date={2022},
			ISSN={0213-2230,2235-0616},
			journal={Rev. Mat. Iberoam.},
			volume={38},
			number={6},
			pages={2011\ndash 2032},
			url={https://doi.org/10.4171/RMI/1323},
			review={\MR{4516180}},
		}
		
		\bib{GRSU20}{article}{
			author={Ghosh, Tuhin},
			author={R\"{u}land, Angkana},
			author={Salo, Mikko},
			author={Uhlmann, Gunther},
			title={Uniqueness and reconstruction for the fractional {C}alder\'{o}n
				problem with a single measurement},
			date={2020},
			ISSN={0022-1236},
			journal={J. Funct. Anal.},
			volume={279},
			number={1},
			pages={108505, 42},
			url={https://doi.org/10.1016/j.jfa.2020.108505},
			review={\MR{4083776}},
		}
		
		\bib{grubb2015fractional}{article}{
			author={Grubb, Gerd},
			title={Fractional {L}aplacians on domains, a development of
				{H}\"ormander's theory of {$\mu$}-transmission pseudodifferential operators},
			date={2015},
			ISSN={0001-8708,1090-2082},
			journal={Adv. Math.},
			volume={268},
			pages={478\ndash 528},
			url={https://doi.org/10.1016/j.aim.2014.09.018},
			review={\MR{3276603}},
		}
		
		\bib{grubb2018green}{article}{
			author={Grubb, Gerd},
			title={Green's formula and a {D}irichlet-to-{N}eumann operator for
				fractional-order pseudodifferential operators},
			date={2018},
			ISSN={0360-5302,1532-4133},
			journal={Comm. Partial Differential Equations},
			volume={43},
			number={5},
			pages={750\ndash 789},
			url={https://doi.org/10.1080/03605302.2018.1475487},
			review={\MR{3920522}},
		}
		
		\bib{GoSc}{article}{
			author={Goldberg, M.},
			author={Schlag, W.},
			title={A limiting absorption principle for the three-dimensional
				{S}chr\"{o}dinger equation with {$L^p$} potentials},
			date={2004},
			ISSN={1073-7928},
			journal={Int. Math. Res. Not.},
			number={75},
			pages={4049\ndash 4071},
			url={https://doi.org/10.1155/S1073792804140324},
			review={\MR{2112327}},
		}
		
		\bib{GSU20}{article}{
			author={Ghosh, Tuhin},
			author={Salo, Mikko},
			author={Uhlmann, Gunther},
			title={The {C}alder\'{o}n problem for the fractional {S}chr\"{o}dinger
				equation},
			date={2020},
			ISSN={2157-5045},
			journal={Anal. PDE},
			volume={13},
			number={2},
			pages={455\ndash 475},
			url={https://doi.org/10.2140/apde.2020.13.455},
			review={\MR{4078233}},
		}
		
		\bib{GU2021calder}{article}{
			author={Ghosh, Tuhin},
			author={Uhlmann, Gunther},
			title={The {C}alder\'{o}n problem for nonlocal operators},
			date={2021},
			journal={arXiv preprint arXiv:2110.09265},
		}
		
		\bib{GuSu}{article}{
			author={Guti\'{e}rrez, Susana},
			title={Non trivial {$L^q$} solutions to the {G}inzburg-{L}andau
				equation},
			date={2004},
			ISSN={0025-5831},
			journal={Math. Ann.},
			volume={328},
			number={1-2},
			pages={1\ndash 25},
			url={https://doi.org/10.1007/s00208-003-0444-7},
			review={\MR{2030368}},
		}
		
		\bib{HYZ}{article}{
			author={Huang, Shanlin},
			author={Yao, Xiaohua},
			author={Zheng, Quan},
			title={Remarks on {$L^p$}-limiting absorption principle of
				{S}chr\"{o}dinger operators and applications to spectral multiplier
				theorems},
			date={2018},
			ISSN={0933-7741},
			journal={Forum Math.},
			volume={30},
			number={1},
			pages={43\ndash 55},
			url={https://doi.org/10.1515/forum-2016-0162},
			review={\MR{3739326}},
		}
		
		\bib{isozaki2014recent}{incollection}{
			author={Isozaki, Hiroshi},
			author={Kurylev, Yaroslav},
			author={Lassas, Matti},
			title={Recent progress of inverse scattering theory on non-compact
				manifolds},
			date={2014},
			booktitle={Inverse problems and applications},
			series={Contemp. Math.},
			volume={615},
			publisher={Amer. Math. Soc., Providence, RI},
			pages={143\ndash 163},
			url={https://doi.org/10.1090/conm/615/12290},
			review={\MR{3221603}},
		}
		
		\bib{MR4058699}{article}{
			author={Ishida, Atsuhide},
			title={Propagation property and application to inverse scattering for
				fractional powers of negative {L}aplacian},
			date={2020},
			ISSN={2079-7362,2079-7370},
			journal={East Asian J. Appl. Math.},
			volume={10},
			number={1},
			pages={106\ndash 122},
			url={https://doi.org/10.4208/eajam.050319.110619},
			review={\MR{4058699}},
		}
		
		\bib{kitada2010scattering}{article}{
			author={Kitada, Hitoshi},
			title={Scattering theory for the fractional power of negative
				{L}aplacian},
			date={2010},
			ISSN={2158-611X},
			journal={J. Abstr. Differ. Equ. Appl.},
			volume={1},
			number={1},
			pages={1\ndash 26},
			review={\MR{2747652}},
		}
		
		\bib{kitada2011remark}{article}{
			author={Kitada, Hitoshi},
			title={A remark on simple scattering theory},
			date={2011},
			ISSN={1938-9787},
			journal={Commun. Math. Anal.},
			volume={11},
			number={2},
			pages={124\ndash 138},
			review={\MR{2780885}},
		}
		
		\bib{KRS}{article}{
			author={Kenig, C.~E.},
			author={Ruiz, A.},
			author={Sogge, C.~D.},
			title={Uniform {S}obolev inequalities and unique continuation for second
				order constant coefficient differential operators},
			date={1987},
			ISSN={0012-7094},
			journal={Duke Math. J.},
			volume={55},
			number={2},
			pages={329\ndash 347},
			url={https://doi.org/10.1215/S0012-7094-87-05518-9},
			review={\MR{894584}},
		}
		
		\bib{kwasnicki2017ten}{article}{
			author={Kwa\'snicki, Mateusz},
			title={Ten equivalent definitions of the fractional {L}aplace operator},
			date={2017},
			ISSN={1311-0454,1314-2224},
			journal={Fract. Calc. Appl. Anal.},
			volume={20},
			number={1},
			pages={7\ndash 51},
			url={https://doi.org/10.1515/fca-2017-0002},
			review={\MR{3613319}},
		}
		
		\bib{lax2014functional}{book}{
			author={Lax, Peter~D.},
			title={Functional analysis},
			series={Pure and Applied Mathematics (New York)},
			publisher={Wiley-Interscience, John Wiley \& Sons, Inc.},
			address={New York},
			date={2002},
			volume={55},
			ISBN={0-471-55604-1},
			review={\MR{1892228}},
		}
		
		\bib{levinson1949inverse}{article}{
			author={Levinson, Norman},
			title={The inverse {S}turm-{L}iouville problem},
			date={1949},
			ISSN={0909-3540},
			journal={Mat. Tidsskr. B},
			volume={1949},
			pages={25\ndash 30},
			review={\MR{32067}},
		}
		
		\bib{LLM21}{article}{
			author={Li, Jingzhi},
			author={Liu, Hongyu},
			author={Ma, Shiqi},
			title={Determining a random {S}chr\"{o}dinger operator: both potential
				and source are random},
			date={2021},
			ISSN={0010-3616},
			journal={Comm. Math. Phys.},
			volume={381},
			number={2},
			pages={527\ndash 556},
			url={https://doi.org/10.1007/s00220-020-03889-9},
			review={\MR{4207450}},
		}
		
		\bib{liu2015determining}{article}{
			author={Liu, Hongyu},
			author={Uhlmann, Gunther},
			title={Determining both sound speed and internal source in thermo- and
				photo-acoustic tomography},
			date={2015},
			ISSN={0266-5611,1361-6420},
			journal={Inverse Problems},
			volume={31},
			number={10},
			pages={105005, 10},
			url={https://doi.org/10.1088/0266-5611/31/10/105005},
			review={\MR{3405365}},
		}
		
		\bib{marchenko2011sturm}{book}{
			author={Marchenko, V.~A.},
			title={Sturm-liouville operators and applications},
			series={Operator Theory: Advances and Applications},
			publisher={Birkh{\"a}user Verlag},
			address={Basel},
			date={1986},
			volume={22},
			ISBN={3-7643-1688-8},
			review={\MR{0481179}},
		}
		
		\bib{nachman1992inverse}{incollection}{
			author={Nachman, Adrian~I.},
			title={Inverse scattering at fixed energy},
			date={1992},
			booktitle={Mathematical physics, {X} ({L}eipzig, 1991)},
			publisher={Springer, Berlin},
			pages={434\ndash 441},
			url={https://doi.org/10.1007/978-3-642-77303-7_48},
			review={\MR{1386440}},
		}
		
		\bib{novikov1988multidimensional}{article}{
			author={Novikov, R.~G.},
			title={A multidimensional inverse spectral problem for the equation
				{$-\Delta\psi +(v(x)-Eu(x))\psi=0$}},
			date={1988},
			ISSN={0374-1990},
			journal={Funktsional. Anal. i Prilozhen.},
			volume={22},
			number={4},
			pages={11\ndash 22, 96},
			url={https://doi.org/10.1007/BF01077418},
			review={\MR{976992}},
		}
		
		\bib{nachman1988n}{article}{
			author={Nachman, Adrian},
			author={Sylvester, John},
			author={Uhlmann, Gunther},
			title={An {$n$}-dimensional {B}org-{L}evinson theorem},
			date={1988},
			ISSN={0010-3616,1432-0916},
			journal={Comm. Math. Phys.},
			volume={115},
			number={4},
			pages={595\ndash 605},
			url={http://projecteuclid.org/euclid.cmp/1104161086},
			review={\MR{933457}},
		}
		
		\bib{ReFr}{article}{
			author={Rellich, Franz},
			title={\"{U}ber das asymptotische {V}erhalten der {L}\"{o}sungen von
				{$\Delta u+\lambda u=0$} in unendlichen {G}ebieten},
			date={1943},
			ISSN={0012-0456},
			journal={Jber. Deutsch. Math.-Verein.},
			volume={53},
			pages={57\ndash 65},
			review={\MR{17816}},
		}
		
		\bib{MR4170189}{article}{
			author={Rakesh},
			author={Salo, Mikko},
			title={Fixed angle inverse scattering for almost symmetric or controlled
				perturbations},
			date={2020},
			ISSN={0036-1410,1095-7154},
			journal={SIAM J. Math. Anal.},
			volume={52},
			number={6},
			pages={5467\ndash 5499},
			review={\MR{4170189}},
		}
		
		\bib{Rakesh_2020}{article}{
			author={Rakesh},
			author={Salo, Mikko},
			title={The fixed angle scattering problem and wave equation inverse
				problems with two measurements},
			date={2020},
			ISSN={0266-5611,1361-6420},
			journal={Inverse Problems},
			volume={36},
			number={3},
			pages={035005, 42},
			url={https://doi.org/10.1088/1361-6420/ab23a2},
			review={\MR{4068234}},
		}
		
		\bib{uhlmann2014uniqueness}{article}{
			author={Rakesh},
			author={Uhlmann, Gunther},
			title={Uniqueness for the inverse backscattering problem for angularly
				controlled potentials},
			date={2014},
			ISSN={0266-5611,1361-6420},
			journal={Inverse Problems},
			volume={30},
			number={6},
			pages={065005, 24},
			url={https://doi.org/10.1088/0266-5611/30/6/065005},
			review={\MR{3224125}},
		}
		
		\bib{MR3372472}{incollection}{
			author={Rakesh},
			author={Uhlmann, Gunther},
			title={The point source inverse back-scattering problem},
			date={2015},
			booktitle={Analysis, complex geometry, and mathematical physics: in honor of
				{D}uong {H}. {P}hong},
			series={Contemp. Math.},
			volume={644},
			publisher={Amer. Math. Soc., Providence, RI},
			pages={279\ndash 289},
			review={\MR{3372472}},
		}
		
		\bib{soccorsi:hal-03571903}{incollection}{
			author={Soccorsi, \'Eric},
			title={Multidimensional {B}org-{L}evinson inverse spectral problems},
			date={2020},
			booktitle={Identification and control: some new challenges},
			series={Contemp. Math.},
			volume={757},
			publisher={Amer. Math. Soc., [Providence], RI},
			pages={19\ndash 49},
			url={https://doi.org/10.1090/conm/757/15248},
			review={\MR{4186958}},
		}
		
		\bib{S}{book}{
			author={Stein, Elias~M.},
			title={Singular integrals and differentiability properties of
				functions},
			series={Princeton Mathematical Series, No. 30},
			publisher={Princeton University Press, Princeton, N.J.},
			date={1970},
			review={\MR{0290095}},
		}
		
		\bib{uhlmann2000inverse}{incollection}{
			author={Uhlmann, Gunther},
			title={Inverse scattering in anisotropic media},
			date={2000},
			booktitle={Surveys on solution methods for inverse problems},
			publisher={Springer, Vienna},
			pages={235\ndash 251},
			review={\MR{1766746}},
		}
		
		\bib{uhlmann2001time}{incollection}{
			author={Uhlmann, Gunther},
			title={A time-dependent approach to the inverse backscattering problem},
			date={2001},
			volume={17},
			pages={703\ndash 716},
			url={https://doi.org/10.1088/0266-5611/17/4/309},
			note={Special issue to celebrate Pierre Sabatier's 65th birthday
				(Montpellier, 2000)},
			review={\MR{1861477}},
		}
		
		\bib{uhlmann2025recovering}{article}{
			author={Uhlmann, Gunther},
			author={Wang, Yiran},
			title={Recovering asymptotics of potentials from the scattering of
				relativistic {S}chr\"{o}dinger operators},
			date={2025},
			journal={arXiv preprint arXiv:2508.12463},
		}
		
		\bib{Zilberberg2026}{article}{
			author={Zilberberg, Dana},
			author={Cakoni, Fioralba},
			author={Vogelius, Michael~S.},
			title={Limiting absorption principle and radiation condition for the
				fractional {H}elmholtz equation},
			date={2026},
			journal={arXiv preprint arXiv:2602.18387},
		}
		
	\end{biblist}
\end{bibdiv}

{


}

\end{document}